\documentclass{amsart}

\usepackage{amsfonts,amssymb,mathrsfs}
\usepackage{graphics,color,hyperref}
\providecommand{\arxiv}[2][]{\href{http://www.arXiv.org/abs/#2}{arXiv:#2}}
\newtheorem{definition}{Definition}

\newtheorem{proposition}{Proposition} 
\newtheorem{lemma}{Lemma} 
\newtheorem{theorem}{Theorem} 
\theoremstyle{remark}
\newtheorem{remark}{Remark}
\newtheorem{example}{Example} 
\title[Duality between invariant spaces for max-plus linear discrete event systems]{Duality between invariant spaces for max-plus linear discrete event systems}
\author{Michael Di Loreto}
\address{Laboratoire Amp\`ere, UMR CNRS 5005, INSA de Lyon, 
20 Avenue Albert Einstein, 69621 Villeurbanne Cedex, France.}
\email{michael.di-loreto@insa-lyon.fr}
\author{St\'ephane Gaubert}
\thanks{The work of the second author was partially supported by the joint 
RFBR-CNRS grant 05-01-02807}
\address{INRIA and Centre de Math\'ematiques Appliqu\'ees, 
\'Ecole Polytechnique. Postal address: CMAP, \'Ecole Polytechnique, 
91128 Palaiseau C\'edex, France. T\`el: +33 1 69 33 46 13, 
Fax: +33 1 39 63 57 86}
\email{Stephane.Gaubert@inria.fr }
\author{Ricardo D. Katz$^*$}
\thanks{$^*$Corresponding author}
\address{CONICET. Postal address: Instituto de Matem\'atica 
``Beppo Levi'', Universidad Nacional de Rosario, 
Avenida Pellegrini 250, 2000 Rosario, Argentina.}
\email{rkatz@fceia.unr.edu.ar}
\author{Jean-Jacques Loiseau}
\address{Institut de Recherche en Communications et en Cybern\'etique de 
Nantes (IRCCyN), UMR CNRS 6597, 1 rue de la Noe BP 92 101, 44 321 Nantes Cedex 3, France.}
\email{Jean-Jacques.Loiseau@irccyn.ec-nantes.fr}

\DeclareMathAlphabet{\mathbbold}{U}{bbold}{m}{n}
\newcommand{\zero}{\varepsilon}
\newcommand{\unit}{0}
\newcommand{\set}[2]{\{#1\mid\,#2\}}
\newcommand{\spancon}[1]{\langle #1 \rangle}
\newcommand{\R}{\mathbb{R}}
\newcommand{\N}{\mathbb{N}}
\newcommand{\Z}{\mathbb{Z}}

\newcommand{\Zm}{\Z\cup\{-\infty\}}

\newcommand{\rmax}{\R_{\max}}

\newcommand{\zmax}{\Z_{\max}}
\newcommand{\sL}{\mathscr{L}}
\newcommand{\sM}{\mathscr{M}}
\newcommand{\sLg}{\mathscr{L}_{\rm fg}}
\newcommand{\sMg}{\mathscr{M}_{\rm fg}}

\newcommand{\mrm}[1]{\text{\rm #1}}

\newcommand{\cC}{\mathcal{C}}
\newcommand{\cB}{\mathcal{B}}
\newcommand{\cV}{\mathcal{V}}

\newcommand{\cU}{\mathcal{U}}
\newcommand{\cS}{\mathcal{S}}
\newcommand{\cW}{\mathcal{W}}
\newcommand{\cX}{\mathcal{X}}
\newcommand{\cY}{\mathcal{Y}}
\newcommand{\cZ}{\mathcal{Z}}
\newcommand{\cK}{\mathcal{K}}
\newcommand{\cKg}{\mathcal{K}^{\rm fg}}
\newcommand{\cVg}{\mathcal{V}_{\rm fg}}
\newcommand{\im}{\mbox{\rm Im}\,}
\newcommand{\vol}{\mrm{vol}\,}

\newcommand{\smallvect}{\mrm{span}\,}

\begin{document}

\begin{abstract}
We extend the notions of conditioned and controlled invariant spaces to 
linear dynamical systems over the max-plus or tropical semiring. 
We establish a duality theorem relating both notions, 
which we use to construct dynamic observers. 
These are useful in situations in which some of the system coefficients 
may vary within certain intervals. 
The results are illustrated by an application to a manufacturing system.  
\end{abstract}

\maketitle

\section{Introduction}

The use of geometric-type techniques when dealing with linear dynamical 
systems, following a line of work initiated by 
Basile and Marro~\cite{BasMar69} and 
Morse and Wonham~\cite{MorWon70,MorWon71}, has provided important 
insights to system-theoretic and control-synthesis problems. 
In particular, this kind of techniques lead to elegant solutions to 
many control problems, such as the disturbance decoupling problem 
and the model matching problem, to quote but a few. 
To achieve this, a geometric approach, using certain linear spaces 
known as conditioned and controlled invariant spaces, 
has been developed (see~\cite{wonham,BasMar91} and the references therein). 

In the classical geometric approach to the theory of linear dynamical systems, 
the scalars belong to a field, or at least to a ring. However, 
the case where the scalars belong to a semiring is also of practical interest. 
In particular, linear dynamical systems with coefficients in the max-plus 
or tropical semiring, and other similar algebraic structures sometimes 
referred to as ``dioids'' or ``idempotent semirings'', 
arise in the modeling and analysis of some manufacturing systems 
following the approach initiated by 
Cohen, Dubois, Quadrat and Viot~\cite{cohen85a} 
(a systematic account can be found in the book by Baccelli, 
Cohen, Quadrat and Olsder~\cite{bcoq}).   
More recent developments of the max-plus approach include the 
``network calculus'' of Le Boudec and Thiran~\cite{leboudec}, 
which can be used to assess certain issues concerning the quality of 
service in telecommunication networks, or an application to train networks 
by Heidergott, Olsder and van der Woude~\cite{how06}. 
All these works provide important examples of discrete event dynamical systems 
subject to synchronization constraints that can be described by 
max-plus linear dynamical systems. 

Several results from linear system theory have been extended to the max-plus 
algebra framework, such as transfer series methods or the connection between 
spectral theory and stability questions (see~\cite{cohen89a}). 
In view of the potentiality of the theory of linear dynamical systems 
over the max-plus semiring, it is also tempting to generalize the geometric 
approach to these systems, a problem which was raised by 
Cohen, Gaubert and Quadrat in~\cite{ccggq99}. 
However, this generalization is not straightforward, 
because many concepts and results must be properly redefined and adapted.
Similar difficulties were already met in the case of linear dynamical 
systems over rings (for which we refer to the works of Hautus~\cite{hautus82} 
and of Conte and Perdon~\cite{conte94,conte95}), 
and in the case of linear systems of infinite dimensions on Hilbert spaces 
(see for instance Curtain~\cite{Cu86}). 
The difficulties are in two directions. 
In the first place, there are algorithmic issues. The concepts of conditioned 
and controlled invariant subspaces (or submodules) 
are no longer dual, and the convergence of the algorithms of the 
geometric approach is not guaranteed. Then, the computation of 
these spaces may be difficult or impossible in general. In the second place, 
the connection between invariance and control or estimation problems 
is more difficult to establish. 
Hypothesis must be added to overcome these problems. 

In this paper, we show that some of the main results of the 
geometric approach do carry over to the max-plus case. 
A first work in this direction was developed by Katz~\cite{katz07}, 
who studied the $(A,B)$-invariant spaces of max-plus linear systems 
providing solutions to some control problems.  
The max-plus analogue of the disturbance decoupling problem has been studied 
by Lhommeau et al.~\cite{LHC03,Lhommeau} making use of invariant sets in the 
spirit of the classical geometric approach. More precisely, 
principal ideal invariant sets were considered, 
which is an elegant solution to the algorithmic issues, 
leading to effective algorithms at the price of a restrictive assumption.
However, these works differ from~\cite{katz07} in the fact that 
they are based on residuation theory and transfer series techniques. 

The present paper is devoted to studying the max-plus analogues of 
conditioned and controlled invariance and the duality between them. 
In the classical linear system theory, conditioned invariant spaces 
are defined in terms of the kernel of the output matrix. 
In the semiring case, the usual definition of the kernel of a matrix is not 
pertinent because it is usually trivial. 
In their places, we consider a natural extension of kernels, 
the congruences, which are equivalence relations with a semimodule 
structure (see~\cite{CGQ96a,CGQ97a,gk08u}). Instead of considering, 
for instance, situations in which the perturbed state $x'$ of the system 
is the sum of the unperturbed state $x$ and of a noise $w$, 
we require the states $x$ and $x'$ to belong to the same equivalence 
class modulo a relation (congruence) which represents the perturbation. 
Indeed, in the max-plus setting, considering only additive 
perturbations would be an important restriction, 
because adding only means delaying events, 
whereas congruences allow us to model situations in which 
the perturbation drives some events to occur at an earlier time. 

By a systematic application of these ideas, 
we generalize the main notions of the classical geometric approach. 
However, this generalization raises new theoretical as well as 
algorithmic issues, because we have to work with congruences 
(sets of pairs of vectors), 
rather than with linear spaces (sets of vectors), 
leading to a general ``doubling'' of the dimension. 

Considering max-plus linear systems subject to perturbations modeled by 
congruences actually leads to an extension of the modeling power of the 
max-plus approach, allowing one to take certain classes of 
constraints or uncertainties into account. For instance, 
we show that max-plus linear dynamical systems 
with uncertain holding times can be modeled in this way, if these times 
are assumed to belong to certain intervals. For this kind of perturbed systems,
the minimal conditioned invariant space containing the perturbation 
can be interpreted as the ``best information'' that can be learned 
on the state of the system from a given observation and initial state.   
Our final result (Theorem~\ref{TheoObserver}) shows 
that this ``optimal information'' on the perturbed state of the system 
can be reconstructed from the output by means of a dynamic observer. 

In order to compute this dynamic observer, 
we extend to the max-plus algebra framework 
the classical fixed point algorithms used to compute the minimal conditioned 
invariant and the maximal controlled invariant spaces containing and 
contained, respectively, in a given space.  
Our main result, Theorem~\ref{TheoCondContr}, establishes a duality between 
conditioned and controlled invariant spaces. 
This allows us to reduce the computation of minimal conditioned invariant 
spaces to the computation of maximal controlled invariant spaces, 
and in this way to reduce algorithmic problems concerning congruences to 
algorithmic problems concerning semimodules, which are easier to handle.    
Thus, this duality theorem solves the previously mentioned algorithmic 
difficulties related to the ``doubling'' of the dimension. 
Then, Proposition~\ref{PropFiniteVolume} identifies conditions 
which guarantee that the fixed point algorithm used to compute 
the minimal conditioned invariant congruence containing 
a given congruence terminates in a finite number of steps. 
Its proof is based on a finite chain condition, 
which is valid thanks to the finiteness and integrity 
assumptions made in the proposition. 
Under more general circumstances, the max-plus case shows difficulties 
which seem somehow reminiscent of the ones encountered in 
the theory of invariant spaces for linear systems over non-Noetherian rings.  
Recall that the computation of such spaces is still an open problem 
in the case of general rings.

This paper is organized as follows. 
The next section is devoted to recalling basic definitions and results on 
max-plus algebra which will be used throughout this paper. 
In Section~\ref{SCondCont} we introduce the max-plus analogues of 
conditioned and controlled invariant spaces and extend to the max-plus algebra
framework the classical fixed point algorithms used for their computation. 
Duality between conditioned and controlled invariance 
is investigated in Section~\ref{SDuality} but previously, 
in Section~\ref{SOrthogonal}, it is convenient to introduce the notions of 
orthogonal of semimodules and congruences and study their properties. 
Finally, in Section~\ref{SApplication}, 
we illustrate the results presented here 
with their application to a manufacturing system.

\section{Preliminaries}\label{SPreliminaries}

The max-plus semiring, $\rmax$, 
is the set $\R\cup\{-\infty\}$ equipped with the 
addition $(a,b)\mapsto \max(a,b)$ and the multiplication $(a,b)\mapsto
a+b$. To emphasize the semiring structure, we write $a\oplus b:=\max(a,b)$
and $ab:=a+b$. 

For $p,q \in \N$, we denote by $\rmax^{p\times q}$ 
the set of all $p$ times $q$ matrices over the max-plus semiring. 
As usual, if $E\in \rmax^{p\times q}$, $E_{ij}$ denotes  
the element of $E$ in its $i$-th row and $j$-th column, 
and $E^t\in \rmax^{q\times p}$ the transposed of $E$. 
The semiring operations are extended in 
the natural way to matrices over the max-plus semiring: 
$(E\oplus F)_{ij}:=E_{ij}\oplus F_{ij}$, 
$(EF)_{ij}:=\oplus_k E_{ik} F_{kj}$ and $(\lambda E)_{ij}:=\lambda E_{ij}$ 
for all $i,j$, where $E$ and $F$ are matrices of compatible dimension and 
$\lambda \in \rmax$. For $E\in \rmax^{p\times q}$, 
we denote by $\im E:=\set{Ex}{x\in \rmax^q}$ the image of $E$.
We usually denote by $\zero:=-\infty$ the neutral element 
for addition as well as the null matrix of any dimension.

We equip $\rmax$ with the usual topology which 
can be defined by the metric: $d(a,b):=|\exp(a)-\exp(b)|$. 
The Cartesian product $\rmax^n$ is equipped with the product topology. 
Note that the semiring operations are continuous 
with respect to this topology.   

The analogues of vector spaces or modules obtained by replacing 
the field or ring of scalars by an idempotent semiring are called  
{\em semimodules} or {\em idempotent spaces}. 
They have been studied by several authors with different motivations
(see for example~\cite{zimmerman77,maslov92,GargKumar95,litvinov00,cgq02}). 
Here, we will only consider subsemimodules of the Cartesian 
product $\rmax^n$, also known as {\em max-plus cones}, 
which are subsets $\cK$ of $\rmax^n$ stable by max-plus linear combinations, 
meaning that
\begin{align*}
\lambda x\oplus \mu y \in \cK
\label{e-stable}
\end{align*}
for all $x,y\in \cK$ and $\lambda,\mu \in \rmax$. 
We denote by $\smallvect \cS$ the smallest semimodule containing a
subset $\cS$ of $\rmax^n$. Therefore, $\smallvect \cS$ is the set of 
all max-plus linear combinations of finitely many elements of $\cS$. 
A semimodule $\cK$ is said to be finitely generated, 
if there exists a finite set $\cS$ such that $\cK=\smallvect \cS$, 
which also means that $\cK=\im E$ for some matrix $E$. 
We shall need the following lemma.

\begin{lemma}[\cite{BSS,gk07}]\label{FGClosed}
Finitely generated subsemimodules of $\rmax^n$ are closed.
\end{lemma} 

A {\em congruence} on $\rmax^n$ is an 
equivalence relation $\cW\subset (\rmax^n)^2$ on $\rmax^n$ which has 
a semimodule structure when it is thought of as a subset of $(\rmax^n)^2$. 
Congruences can be seen as the max-plus analogues of kernels of the 
classical theory: due to the absence of minus sign, 
given a matrix $E\in \rmax^{p\times n}$, it is natural to define 
the {\em kernel} of $E$ (see~\cite{CGQ96a,CGQ97a}) as the congruence
\[  
\ker E:=\set{(x,y)\in (\rmax^n)^2}{E x = E y} \; . 
\] 
The usual definition $\ker E:=\set{x\in \rmax^n}{E x = \zero }$ 
is not convenient in the max-plus algebra case, 
because this semimodule is usually trivial even if $E$ is not injective, 
so it carries little information. 
If $A\in \rmax^{n\times n}$ is a matrix and $\cW\subset (\rmax^n)^2$ 
is a congruence, we define 
\[
A\cW:=\set{(Ax,Ay)\in (\rmax^n)^2}{(x,y)\in \cW}
\] 
and 
\[
A^{-1}\cW:=\set{(x,y)\in (\rmax^n)^2}{(Ax,Ay)\in \cW} \; . 
\]
Observe that $A\cW$ is not necessarily a congruence even if $\cW$ is. 
For $\cS \subset \rmax^n$, we define as usual 
\[ 
A\cS:=\set{Ax\in \rmax^n}{x\in \cS} \;
\makebox{ and} \; 
A^{-1}\cS:=\set{x\in \rmax^n}{Ax\in \cS} \; .
\] 
In the sequel, if $\cW$ is a congruence, we write $x\sim_\cW y$
for $(x,y)\in \cW$ and denote by $[x]_{\cW}$ the equivalence
class of $x$ modulo $\cW$.

\section{Max-plus conditioned and controlled invariance}\label{SCondCont}

We consider max-plus dynamical systems of the form
\begin{equation}\label{e-fond}
\left\{ 
\begin{array}{l}
x(k+1) \sim_\cV Ax(k)  \\
y(k)=Cx(k) \\
x(0)=x
\end{array}
\right. 
\end{equation}
where $A\in \rmax^{n\times n}$, $C\in \rmax^{q\times n}$, $x\in \rmax^n$ 
is the initial state, $x(k)\in \rmax^n$ is the state, 
$y(k)\in \rmax^q$ is the output and $\cV\subset (\rmax^n)^2$ 
is a congruence which represents unobservable perturbations 
of the max-plus linear system $x(k+1)=Ax(k)$.
Hence, the dynamics in~\eqref{e-fond} is multi-valued, 
meaning that several values of $x(k+1)$ are compatible with a given $x(k)$.
We assume that the output $y(k)$ is observed. 

If $\cV=\ker E$ for some matrix $E$, then 
$x(k+1)\sim_{\cV} A x(k)$ is equivalent to $E x(k+1)=E A x(k)$,
so $x\mapsto E x$ may be interpreted as an invariant which must
be preserved by the perturbation. Hence, system~\eqref{e-fond} might
be viewed as an implicit linear system. In classical system theory,
implicit systems are often used to represent systems subject to disturbances.

In Section~\ref{SApplication} we will show that dynamical systems of the 
form~\eqref{e-fond} can be used, for instance, 
to model max-plus linear dynamical systems of the form 
\[
x(k+1)=\bar{A} x(k)\; , 
\]
when some entries of $\bar{A}$ are unknown but belong to certain intervals.

The analogy between congruences and classical kernels leads 
us to the following definition.  

\begin{definition}[Conditioned Invariant]
Given $A\in \rmax^{n\times n}$ and $C\in \rmax^{q\times n}$, 
a congruence $\cW\subset (\rmax^n)^2$ is said to be 
$(C,A)$-conditioned invariant if 
\begin{equation}
A(\cW \cap \cC)\subset \cW \; , 
\end{equation}
where $\cC:=\ker C$.
\end{definition}

The following proposition establishes the connection 
between conditioned invariants and the observation problem
for dynamical systems of the form~\eqref{e-fond}. 

\begin{proposition}\label{CondInvInterp}
Let $\cW$ be a congruence containing the perturbation $\cV$. 
Then, $\cW$ is $(C,A)$-conditioned invariant if, and only if, 
for any trajectory $\left\{ x(k)\right\}_{k\geq 0}$ of 
system~\eqref{e-fond} and any $m\in \N$,  
the equivalence class of $x(m)$ modulo $\cW$ is uniquely 
determined by the equivalence class of $x(0)$ modulo $\cW$
and by the observations $y(0),\ldots,y(m-1)$. 
\end{proposition}

\begin{proof}
Assume that $\cW$ is $(C,A)$-conditioned invariant.  
Let $x(k+1)\sim_\cV Ax(k)$ and $x'(k+1)\sim_\cV Ax'(k)$.
Then, if $x(k)\sim_\cW x'(k)$ and $y(k):=Cx(k)=y'(k):=Cx'(k)$, 
we have $(x(k),x'(k))\in \cW\cap \cC$, and so 
$(Ax(k),Ax'(k))\in A(\cW\cap \cC)\subset \cW$ because 
$\cW$ is $(C,A)$-conditioned invariant and $\cC=\ker C$. 
Therefore, $Ax(k)\sim_\cW Ax'(k)$, and since $\cW\supset \cV$, 
we deduce that $x(k+1)\sim_\cW x'(k+1)$. 
The ``only if'' part of the proposition follows from an immediate induction.

Conversely, assume that for any trajectory $\left\{ x(k)\right\}_{k\geq 0}$ 
of system~\eqref{e-fond} the equivalence class of $x(m)$ modulo 
$\cW$ is uniquely determined by the equivalence class of $x(0)$ modulo $\cW$
and by the observations. Let  $(x(0),x'(0))\in \cW\cap \cC$. Then, 
if $x(1)\sim_\cV Ax(0)$ and $x'(1)\sim_\cV Ax'(0)$,  
we have $x(1)\sim_\cW x'(1)$ because $x(0)\sim_\cW x'(0)$ and 
$y(0):=Cx(0)=y'(0):=Cx'(0)$. Therefore, 
it follows that $Ax(0) \sim_\cW Ax'(0)$ because $\cV \subset \cW$. 
Since this holds for any $(x(0),x'(0))\in \cW\cap \cC$, 
we conclude that $A (\cW\cap \cC) \subset \cW$, 
which proves the ``if'' part of the proposition. 
\end{proof}
 
This proposition raises, for observation purpose, 
the question of the existence, and the computation when it exists, 
of the minimal $(C,A)$-conditioned invariant congruence containing $\cV$. 
Like in the case of coefficients in a field, 
the following lemma can be easily proved. 
 
\begin{lemma}\label{InterConditioned}
The intersection of $(C,A)$-conditioned invariant congruences 
is a $(C,A)$-conditioned invariant congruence. 
\end{lemma}

If we denote by $\sL(C,A,\cV)$ the set of all $(C,A)$-conditioned 
invariant congruences containing a given congruence $\cV$, then, 
as a consequence of the previous lemma, it follows that 
$\sL(C,A,\cV)$ is a lower semilattice with respect to $\subset$ and $\cap$. 
Moreover, Lemma~\ref{InterConditioned} also implies that $\sL(C,A,\cV)$ 
admits a smallest element, the minimal $(C,A)$-conditioned 
invariant congruence containing $\cV$, which will be denoted 
by $\cV_*(C,A)$.

In order to compute $\cV_*(C,A)$, 
we extend the classical fixed point algorithm 
(see~\cite{BasMar69,BasMar91,wonham}) to the max-plus algebra framework. 
With this purpose in mind, 
consider the self-map $\psi$ of the set of congruences given by 
\begin{equation}\label{DefPsi}
\psi (\cW):=\spancon {\cV\oplus A(\cW \cap \cC)} \enspace, 
\end{equation}
where $\spancon \cU$ denotes the smallest congruence containing the set 
$\cU \subset (\rmax^n)^2$.  
Note that $\cW \subset \cU$ implies $\psi (\cW )\subset \psi (\cU )$. 
Define the sequence of congruences $\left\{\cW_k\right\}_{k\in \N}$ by:
\begin{equation}\label{DefSequenceW}
\cW_1 := \cV \; \makebox{ and } \; 
\cW_{k+1}:=\psi (\cW_{k}) \;  \makebox{ for } \; k\in \N \; .
\end{equation} 
Then, this sequence is (weakly) increasing, that is, 
$\cW_k\subset \cW_{k+1}$ for all $k\in \N$. As a matter of fact, 
$\cW_1=\cV \subset \cV\oplus A(\cW_1 \cap \cC)\subset 
\spancon{\cV\oplus A(\cW_1 \cap \cC)}=\cW_2$ and if 
$\cW_r\subset \cW_{r+1}$, then 
$\cW_{r+1}=\psi (\cW_r)\subset \psi (\cW_{r+1})=\cW_{r+2}$.
We define $\cV_\infty$ as the limit of the sequence 
$\left\{\cW_k\right\}_{k\in \N}$, that is, $\cV_\infty=\cup_{k\in \N}\cW_k$. 
Note that $\cV_\infty$ is a congruence, because 
$\left\{\cW_k\right\}_{k\in \N}$ 
is an increasing sequence of congruences. 

\begin{proposition}\label{FixedPointCong}
Let $\cV\subset (\rmax^n)^2$ be a congruence. Then, $\cV_\infty$ is 
the minimal $(C,A)$-conditioned invariant congruence containing $\cV$. 
\end{proposition}

\begin{proof}
Let $\cW\subset (\rmax^n)^2$ be a $(C,A)$-conditioned invariant congruence 
containing $\cV$. We next show that $\cW_k\subset \cW$ for all $k\in \N$, 
and therefore $\cV_\infty\subset \cW$. In the first place, note that 
$\cW_1=\cV \subset \cW$. Assume now that $\cW_r \subset \cW$. Then, as 
$A(\cW \cap \cC)\subset \cW$ and $\cV\subset \cW$, it follows that  
$\cW_{r+1}=\psi(\cW_r)\subset \psi(\cW)=\spancon{\cV\oplus A(\cW \cap \cC)}
\subset \spancon \cW= \cW$.  
 
To prove that $\cV_*(C,A)=\cV_\infty$, it only remains to show that 
$\cV_\infty$ is $(C,A)$-conditioned invariant. Since
\[ 
A(\cV_\infty\cap \cC)=A((\cup_k\cW_k)\cap \cC)=A(\cup_k (\cW_k\cap \cC))= 
\cup_k(A(\cW_k\cap \cC))\subset \cup_k\cW_{k+1}=\cV_\infty \; ,
\]
it follows that $\cV_\infty$ is a $(C,A)$-conditioned invariant congruence.
\end{proof}

Concerning the computation of $\cV_*(C,A)$, 
Proposition~\ref{FixedPointCong} presents two drawbacks in relation 
to the classical theory. In the first place, for linear systems over fields, 
the sequence $\left\{\cW_k\right\}_{k\in \N}$ always converges in at 
most $n$ steps because it is an increasing sequence of subspaces 
of a vector space of dimension $n$. However, in the max-plus case, 
this sequence does not necessarily converge in a finite number of steps 
(see the example below). This difficulty is mainly due to the fact that 
$(\rmax^n)^2$ is not Noetherian, meaning that there exist infinite 
increasing sequences of subsemimodules of $(\rmax^n)^2$. 
The second difficulty comes from the fact the 
$\cV\oplus A(\cW_k \cap \cC)$ need not be a congruence, 
so it is necessary to compute $\spancon{ \cV\oplus A(\cW_k \cap \cC)}$. 
However, the duality results established in the present paper will allow us to 
dispense with this operation.
 
\begin{example}\label{ExampleSeqW}
Consider the matrices 
\[
A =
\begin{pmatrix}
\unit & \zero \cr
\zero & 1 \cr 
\end{pmatrix}
\; \makebox{ and }\; 
C =
\begin{pmatrix}
\zero & \zero  
\end{pmatrix}
\; ,
\]
and the congruence $\cV \subset (\rmax^2)^2$ defined by: 
$x\sim_\cV y$ if, and only if, $x_1=y_1$ and 
$x_1\oplus x_2 = y_1\oplus y_2$. In order to determine 
the minimal $(C,A)$-conditioned invariant congruence containing $\cV$, 
we next compute the sequence of congruences 
$\left\{\cW_k\right\}_{k\in \N}$ defined 
in~\eqref{DefPsi} and~\eqref{DefSequenceW}. 
We claim that $\cW_k$ is defined as follows: 
$x\sim_{\cW_k} y$ if, and only if, $x_1=y_1$ and 
$(k-1) x_1\oplus x_2 =(k-1) y_1\oplus y_2$. In the first place, 
note that this property is satisfied by definition for $k=1$. 
Assume now that it holds for $k=m$. Note that  
\[
x_1=y_1 \; \makebox{ and } \; 
(m-1) x_1\oplus x_2 =(m-1) y_1\oplus y_2 
\] 
is equivalent to 
\[ 
\left( x_1=y_1 , (m-1) x_1\geq x_2 , (m-1) y_1\geq y_2 \right) \; 
\makebox{ or } \; 
\left( x_1=y_1 , x_2=y_2 \right) \; .
\] 
Then, in this particular case $A\cW_m$ is a congruence which is defined by 
\[ 
x\sim_{A\cW_m} y \iff 
\left( x_1=y_1 , m x_1\geq x_2 , m y_1\geq y_2 \right) \; 
\makebox{ or } \; 
\left( x_1=y_1 , x_2=y_2 \right) 
\] 
and thus 
\[
\cW_{m+1}=\psi (\cW_{m})= 
\spancon {\cV\oplus A(\cW_m \cap \cC)}
=\spancon {\cV \oplus A \cW_m}
=\spancon { A \cW_m} = A \cW_m
\]
because $\cC=\ker C=(\rmax^2)^2$ and $\cV \subset A\cW_m$. 
This proves our claim.

Therefore, $\cV_*(C,A)=\cV_\infty$ is the congruence defined as follows: 
\[
x\sim_{\cV_\infty } y \iff 
\left( x_1 = y_1\neq \zero \right) \; \makebox{ or } \; 
\left( x_1 = y_1 = \zero , x_2 = y_2 \right) \; .
\]
Note that $\cV_*(C,A)=\cV_\infty$ is not closed even if $\cV$ is closed. 
For instance, if $\lambda_1\neq \lambda_2$, we have 
\[
(-k,\lambda_1)^t\sim_{\cV_\infty } (-k,\lambda_2)^t
\] 
for all $k\in \N$, but 
$(\zero ,\lambda_1)^t {\not \sim}_{\cV_\infty } (\zero,\lambda_2)^t$.
\end{example}

For linear systems over fields, the minimal $(C,A)$-conditioned 
invariant space containing a given space can be alternatively 
computed through the notion of controlled invariance, 
which is dual of the notion of conditioned invariance.  
In the max-plus case, this dual notion can be defined as follows.

\begin{definition}[Controlled Invariant]\label{DefControlled}
Given $A\in \rmax^{n\times n}$ and $B\in \rmax^{n\times q}$, 
a semimodule $\cX\subset \rmax^n$ is said to be 
$(A,B)$-controlled invariant if 
\begin{equation}A\cX\subset \cX\oplus \cB \; , 
\end{equation}
where $\cB:=\im B$ and 
$\cX\oplus \cB:=\set{x\oplus b}{x\in \cX ,b\in \cB}$.
\end{definition}

\begin{remark}
From a dynamical point of view, the interpretation of 
$(A,B)$-controlled invariance differs from the classical one. 
For linear dynamical systems over fields of the form 
\begin{equation}\label{ABSystem}
x(k+1)=A x(k) + B u(k) \; ,
\end{equation}
where $x(k)$ is the state, $u(k)$ is the control, 
and $A$ and $B$ are matrices of suitable dimension, 
it can be shown (see~\cite{BasMar91,wonham}) that $\cX$ is 
$(A,B)$-controlled invariant if, and only if, 
any trajectory of~\eqref{ABSystem} starting in $\cX$ can 
be kept inside $\cX$ by a suitable choice of the control. 
However, due to the non-invertibility of addition, 
this is no longer true in the max-plus case. 
For this property to hold true, 
in Definition~\ref{DefControlled} the semimodule 
$\cX\oplus \cB$ must be replaced by 
$\cX \ominus \cB := \set{z\in \rmax^n}{\exists b\in \cB , z\oplus b\in \cX}$ 
(see~\cite{katz07} for details).
\end{remark}

The proof of the following simple lemma, which is dual 
of Lemma~\ref{InterConditioned}, is left to the reader. 
 
\begin{lemma}\label{SumControlled}
The (max-plus) sum of $(A,B)$-controlled invariant 
semimodules is $(A,B)$-controlled invariant. 
\end{lemma}

By Lemma~\ref{SumControlled} the set of all 
$(A,B)$-controlled invariant semimodules contained in a given semimodule 
$\cK\subset \rmax^n$, which will be denoted by $\sM(A,B,\cK)$, 
is an upper semilattice with respect to $\subset$ and $\oplus$. 
In this case, $\sM(A,B,\cK)$ admits a biggest element, 
the maximal $(A,B)$-controlled invariant 
semimodule contained in $\cK$, which will be denoted by $\cK^*(A,B)$. 
 
In order to compute $\cK^*(A,B)$, 
consider the self-map $\phi$ of the set of semimodules defined by: 
\begin{equation}\label{DefPhiX}
\phi (\cX) := \cK\cap A^{-1}(\cX \oplus \cB) \; .
\end{equation} 
Define the sequence of semimodules 
$\left\{\cX_k\right\}_{k\in \N}$ as follows: 
\begin{equation}\label{DefSequenceX} 
\cX_1 := \cK \; \makebox{ and } \; 
\cX_{k+1} := \phi(\cX_{k}) \; \makebox{ for }\; k\in \N \; . 
\end{equation}
Note that $\left\{\cX_k\right\}_{k\in \N}$ is (weakly) decreasing, that is, 
$\cX_{k+1}\subset \cX_k$ for all $k\in \N$. As a matter of fact,   
$\cX_2= \phi(\cX_1)= \cK\cap A^{-1}(\cX_1 \oplus \cB)\subset \cK=\cX_1$ 
and if $\cX_{r+1}\subset \cX_r$, then 
$ \cX_{r+2}=\phi(\cX_{r+1})\subset \phi(\cX_r)=\cX_{r+1}$,  
since $\phi(\cZ)\subset \phi(\cY)$ whenever $\cZ \subset \cY$. 
We define the semimodule $\cK^\infty$ as the limit of the sequence 
$\left\{\cX_k\right\}_{k\in \N}$, that is, 
$\cK^\infty =\cap_{k\in \N}\cX_{k}$. 

\begin{lemma}\label{ABsimple}
Any $(A,B)$-controlled invariant semimodule 
contained in $\cK$ is contained in $\cK^\infty$. 
In particular, $\cK^*(A,B)\subset \cK^\infty$.
\end{lemma}
\begin{proof}
Let $\cX$ be an $(A,B)$-controlled invariant semimodule contained in $\cK$. 
We next prove (by induction on $k$) that $\cX\subset \cX_k$ for all $k\in \N$, 
and thus $\cX\subset \cK^\infty$. In the first place, note that 
$\cX\subset \cK=\cX_{1}$. Assume now that $\cX\subset \cX_r$. Then, 
as $A\cX\subset \cX\oplus \cB $ and $\cX\subset \cK$, it follows that 
$\cX\subset \cK\cap A^{-1}(\cX\oplus \cB) = 
\phi(\cX)\subset \phi(\cX_r)= \cX_{r+1}$. 
\end{proof}
 
In the sequel, we will repeatedly use the following elementary observation. 

\begin{lemma}\label{SumClosed}
If $\cX$ and $\cY$ are closed subsemimodules
of $\rmax^n$, then so is $\cX\oplus \cY$.
\end{lemma}

\begin{proof}
Let $\left\{z_k\right\}_{k\in \N}$ 
denote a sequence of elements of $\cX\oplus \cY$
converging to some $z\in \rmax^n$. Then, 
we can write $z_k=x_k\oplus y_k$ with $x_k\in \cX$ and $y_k\in \cY$ 
for $k\in \N$. Since $\left\{z_k\right\}_{k\in \N}$ is bounded,
$\left\{x_k\right\}_{k\in \N}$ and $\left\{y_k\right\}_{k\in \N}$ 
must be bounded, and so, by taking subsequences
if necessary, we may assume that $\left\{x_k\right\}_{k\in \N}$ and 
$\left\{y_k\right\}_{k\in \N}$ converge to some vectors
$x$ and $y$, respectively. Since $\cX$ and $\cY$ are closed,
we have $x\in \cX$ and $y\in \cY$, and so, $z=x\oplus y\in \cX\oplus \cY$.
\end{proof}

In order to state a dual of Proposition~\ref{FixedPointCong}, 
we shall need a topological assumption. 

\begin{proposition}\label{Kclosed}
Let $\cK\subset \rmax^n$ be a closed semimodule. Then, $\cK^\infty$ is the 
maximal $(A,B)$-controlled invariant semimodule contained in $\cK$.
\end{proposition}

\begin{proof}
By Lemma~\ref{ABsimple}, it suffices to show that $\cK^\infty$ 
is $(A,B)$-controlled invariant, that is, 
$A\cK^\infty\subset \cK^\infty\oplus \cB$. With this aim, as 
$A\cK^\infty = A(\cap_k\cX_{k+1}) \subset 
\cap_k A\cX_{k+1}\subset \cap_k (\cX_k\oplus \cB)$, 
it is enough to prove that 
$\cap_k (\cX_k\oplus \cB)\subset (\cap_k\cX_k)\oplus \cB = 
\cK^\infty\oplus \cB$. 
 
In the first place, note that by Lemma~\ref{SumClosed}, 
$\phi(\cX)$ is closed whenever $\cX$ and $\cK$ are closed, 
because $\cB$ is closed by Lemma~\ref{FGClosed}. Then, 
the semimodules $\cX_k$ are all closed since $\cK$ is closed. 
If $x\in \cap_k (\cX_k\oplus \cB)$, there exist sequences 
$\left\{b_k\right\}_{k\in \N}$ and $\left\{x_k\right\}_{k\in \N}$ 
such that $x=x_k\oplus b_k$, $x_k\in \cX_k$ and 
$b_k\in \cB$ for all $k\in \N$. 
As these sequences are bounded by $x$, we may assume, by taking 
subsequences if necessary, that there exist $y\in \rmax^n$ and 
$b\in \cB$ such that $\lim_{k\rightarrow\infty}x_k=y$ and 
$\lim_{k\rightarrow\infty}b_k=b$ (recall that $\cB$ is closed 
by Lemma~\ref{FGClosed}). Then, 
as the sequence $\left\{\cX_k\right\}_{k\in \N}$ 
is decreasing and the semimodules $\cX_k$ are all closed, 
it follows that $y=\lim_{k\rightarrow\infty}x_k\in \cX_r$ for all $r\in \N$. 
Therefore, $y\in \cap_k \cX_k=\cK^\infty$ and 
$x=\lim_{k\rightarrow\infty}(x_k\oplus b_k)=
(\lim_{k\rightarrow\infty}x_k)\oplus (\lim_{k\rightarrow\infty}b_k)=
y\oplus b\in \cK^\infty\oplus \cB$.
\end{proof}
 
Observe that, by Lemma~\ref{FGClosed}, the condition of the previous 
proposition is in particular satisfied when $\cK$ is finitely generated. 
Note also that $\cK^*(A,B)=\cK^\infty$ is closed if $\cK$ is closed, 
because in that case $\cK^\infty$ is an intersection of closed semimodules 
(recall that in the previous proof we showed that the semimodules 
$\cX_k$ are all closed when $\cK$ is closed). 

Like in the case of the sequence of congruences 
$\left\{\cW_k\right\}_{k\in \N}$, and unlike the case 
of coefficients in a field in which it converges in at most 
$n$ steps (see~\cite{BasMar69,BasMar91,wonham}), 
the sequence of semimodules $\left\{\cX_k\right\}_{k\in \N}$ 
does not necessarily converge in a finite number of steps 
(see the example below). 
This is in part a consequence of the fact that $\rmax^n$ is not Artinian, 
meaning that there exist infinite decreasing sequences 
of subsemimodules of $\rmax^n$. However, 
in Section~\ref{SDuality} we will give a condition which ensures 
the convergence of this sequence in a finite number of steps. 
This difficulty is also found when the coefficients belong to a ring, 
where except for Principal Ideal Domains, 
the computation of the maximal $(A,B)$-controlled invariant module
is still under investigation (see~\cite{conte94,conte95}).

\begin{example}\label{ExampleSeqX}
Consider the matrices 
\[
A =
\begin{pmatrix}
\unit & \zero \cr
\zero & 1 \cr 
\end{pmatrix}
\; \makebox{ and }\; 
B =
\begin{pmatrix}
\zero \cr 
\zero  
\end{pmatrix}
\; ,
\]
and the semimodule $\cK =\set{x\in \rmax^2}{x_1\geq x_2}$. 
Since $\cK$ is clearly closed, we can apply Proposition~\ref{Kclosed} 
in order to compute $\cK^*(A,B)$. If we define the sequence of 
semimodules $\left\{\cX_k\right\}_{k\in \N}$ by~\eqref{DefPhiX} 
and~\eqref{DefSequenceX}, then, using the fact that in this 
particular case $A$ is invertible and that  
\[
A^{-1} =
\begin{pmatrix}
\unit & \zero \cr
\zero & -1 \cr 
\end{pmatrix} \; ,  
\]
it can be easily seen that $\cX_k=\set{x\in \rmax^2}{x_1\geq (k-1)x_2}$ 
for all $k\in \N$. Therefore, 
$\cK^*(A,B)=\cK^\infty=\set{x\in \rmax^2}{x_2=\zero}$.
\end{example}

\section{Orthogonal semimodules and congruences}\label{SOrthogonal}

Before studying the duality between controlled and conditioned invariance, 
it is convenient to introduce the notions of orthogonal of 
semimodules and congruences, and study their properties.

\begin{definition}
The orthogonal of a semimodule $\cX\subset \rmax^n$ is the congruence 
$\cX^\bot =\set{(x,y)\in (\rmax^n)^2}{x^tz=y^tz,\forall z\in \cX}$. 
Analogously, the orthogonal of a congruence (or more generally a semimodule) 
$\cW\subset (\rmax^n)^2$ is the semimodule 
$\cW^\top =\set{z\in \rmax^n}{x^tz=y^tz,\forall (x,y)\in \cW}$.
\end{definition}

Note that the orthogonal, being the intersection of 
closed sets, is always closed. We shall need the following duality theorem.

\begin{theorem}\label{SepTheo}
If $\cX\subset \rmax^n$ is a closed semimodule, and if
$\cW\subset (\rmax^n)^2$ is a closed congruence, then:
\[
\cX=(\cX^\bot)^\top,\; \makebox{ and }
\; \cW=(\cW^\top)^\bot \; .
\]
\end{theorem}

The first equality follows from the separation theorem
for closed semimodules, see~\cite[Th.~4]{zimmerman77},~\cite{shpiz}, 
see also~\cite[Th.~3.14]{cgqs04} for recent improvements. The second equality
is proved in~\cite{gk08u} as a consequence of a new separation theorem,
which applies to closed congruences.

The orthogonal has the following properties. 

\begin{lemma}\label{propiedad}
Let $A\in \rmax^{n\times n}$ be a matrix, 
$\cW,\cW_1,\cW_2 \subset (\rmax^n)^2$ be congruences and 
$\cX,\cX_1,\cX_2\subset \rmax^n$ be semimodules. 
Then, 
\begin{itemize}
\item[(i)]\label{ph1} $(\cW_1\oplus \cW_2)^\top= \cW_1^\top \cap \cW_2^\top$ 
\makebox{ and } \; $(\cX_1\oplus \cX_2)^\bot= \cX_1^\bot \cap \cX_2^\bot$  ,
\item[(ii)]\label{ph2} $(A\cW)^\top= (A^t)^{-1}\cW^\top$ 
\makebox{ and } \; $(A\cX)^\bot= (A^t)^{-1}\cX^\bot$ .
\end{itemize}
Moreover, if $\cW_1$, $\cW_2$, $\cX_1$ and $\cX_2$ are closed, then
\begin{itemize}
\item[(iii)]\label{ph3} $(\cW_1\cap \cW_2)^\top= \cW_1^\top \oplus \cW_2^\top$
\makebox{ and } \; $(\cX_1\cap \cX_2)^\bot= \cX_1^\bot \oplus \cX_2^\bot$ .
\end{itemize}
\end{lemma}

\begin{proof}

We next prove these properties for congruences. In the case of semimodules, 
these properties can be proved along the same lines. 

(i) As $\cW_r\subset \cW_1\oplus \cW_2$ for $r=1,2$, we have 
$(\cW_1\oplus \cW_2)^\top\subset \cW_r^\top$ for $r=1,2$ and thus 
$(\cW_1\oplus \cW_2)^\top\subset \cW_1^\top \cap \cW_2^\top$. 

Let $z\in \cW_1^\top \cap \cW_2^\top$. Since 
$$(x_1\oplus x_2)^t z = x_1^t z\oplus x_2^t z = 
y_1^t z \oplus y_2^t z = (y_1\oplus y_2)^t z$$
for all $(x_1,y_1)\in \cW_1$ and $(x_2,y_2)\in \cW_2$, 
it follows that $z\in (\cW_1\oplus \cW_2)^\top$. Therefore, 
$\cW_1^\top \cap \cW_2^\top \subset (\cW_1\oplus \cW_2)^\top$. 

(ii) We have
\begin{eqnarray*}
(A\cW)^\top & = & \set{z\in \rmax^n}{x^tz=y^tz,\forall (x,y)\in A\cW} \\
& = & \set{z\in \rmax^n}{(Ax)^tz=(Ay)^tz,\forall (x,y)\in \cW} \\
& = & \set{z\in \rmax^n}{x^tA^tz=y^tA^tz,\forall (x,y)\in \cW} \\
& = & \set{z\in \rmax^n}{A^tz\in \cW^\top}=(A^t)^{-1}\cW^\top\; .
\end{eqnarray*}
 
(iii) Since $\cW_1$ and $\cW_2$ are closed, by Theorem~\ref{SepTheo} 
we have $\cW_1=(\cW_1^\top )^\bot $ and $\cW_2=(\cW_2^\top)^\bot$. 
Then, from (i) and Theorem~\ref{SepTheo}, it follows that
\[ 
(\cW_1\cap \cW_2)^\top= ((\cW_1^\top)^\bot \cap (\cW_2^\top)^\bot)^\top= 
((\cW_1^\top \oplus \cW_2^\top)^\bot)^\top= \cW_1^\top \oplus \cW_2^\top \; , 
\]
because $\cW_1^\top \oplus \cW_2^\top$ is closed by Lemma~\ref{SumClosed}. 
\end{proof}

In Property (iii) above, when the semimodules $\cX_1$ and $\cX_2$ 
are not closed, the only thing that can be said is that
\[ 
\cX_1^\bot \oplus \cX_2^\bot \subset (\cX_1\cap \cX_2)^\bot \; .
\]
As a matter of fact, since $\cX_1\cap \cX_2\subset \cX_r$ for $r=1,2$, 
it follows that $\cX_r^\bot \subset (\cX_1\cap \cX_2)^\bot$ for $r=1,2$ 
and so $\cX_1^\bot \oplus \cX_2^\bot \subset (\cX_1\cap \cX_2)^\bot$. 
To see that the other inclusion does not necessarily hold, 
consider the semimodules $\cX_1=\set{x\in \rmax^2}{x_1=x_2}$ and 
$\cX_2=\set{x\in \rmax^2}{x_1>x_2}\cup \left\{ (\zero, \zero)^t\right\}$. 
Then, $\cX_1^\bot=\set{(x,y)\in (\rmax^2)^2}{x_1\oplus x_2=y_1\oplus y_2}$
and 
$\cX_2^\bot=\set{(x,y)\in (\rmax^2)^2}{x_1=y_1,x_1\oplus x_2=y_1\oplus y_2}$,
thus $\cX_1^\bot \oplus \cX_2^\bot=\cX_1^\bot \varsubsetneq (\rmax^2)^2$ 
because $\cX_2^\bot \subset \cX_1^\bot $. However, 
$(\cX_1\cap \cX_2)^\bot=\left\{ (\zero, \zero)^t\right\}^\bot =(\rmax^2)^2$. 

\begin{lemma}\label{LemmaFG}
For any matrix $E$ we have 
$(\im E)^\bot=\ker E^t$ 
and $(\ker E)^\top =\im E^t$.
\end{lemma}

\begin{proof}
Note that $(x,y)\in (\im E)^\bot \iff x^tz=y^tz, \forall z\in \im E 
\iff x^tEv=y^tEv, \forall v \iff E^tx=E^ty \iff (x,y)\in \ker E^t$. 
Therefore, $(\im E)^\bot =\ker E^t$. 

Since by Lemma~\ref{FGClosed} $\im E^t$ is closed, we have 
$((\im E^t)^\bot)^\top =\im E^t$. Then, as $(\im E^t)^\bot =\ker E$, 
it follows that $(\ker E)^\top =\im E^t$.
\end{proof}

\section{Duality between conditioned and controlled invariance}\label{SDuality}

In this section we investigate the duality between controlled 
and conditioned invariants in max-plus algebra.

\begin{lemma}\label{conditionedcontrolled}
If $\cX\subset \rmax^n$ is $(A,B)$-controlled invariant, 
then $\cX^\bot$ is $(B^t,A^t)$-conditioned invariant. 
Moreover, if a closed congruence $\cW\subset (\rmax^n)^2$ 
is $(C,A)$-conditioned invariant, then 
$\cW^\top$ is $(A^t,C^t)$-controlled invariant.
\end{lemma}
\begin{proof}
If $\cX$ is $(A,B)$-controlled invariant, 
then $A\cX\subset \cX \oplus \im B$. Since 
\begin{eqnarray*}
A\cX\subset \cX \oplus \im B &\implies &
(\cX \oplus \im B)^\bot \subset (A\cX)^\bot \\ 
&\implies & \cX^\bot \cap (\im B)^\bot \subset (A^t)^{-1}\cX^\bot \\
&\implies & A^t (\cX^\bot \cap \ker B^t)\subset \cX^\bot, 
\end{eqnarray*}
it follows that $\cX^\bot$ is $(B^t,A^t)$-conditioned invariant. 

Assume that a closed congruence $\cW$ is $(C,A)$-conditioned invariant, 
that is $A (\cW \cap \ker C)\subset \cW$. Then, by Lemma~\ref{propiedad} 
we have $(\cW \cap \ker C)^\top =\cW^\top \oplus (\ker C)^\top$ 
and thus
\begin{eqnarray*}
A (\cW \cap \ker C)\subset \cW &\implies &
\cW^\top \subset (A(\cW \cap \ker C))^\top \\  
& \implies & \cW^\top \subset (A^t)^{-1} (\cW \cap \ker C)^\top\\
&\implies & A^t\cW^\top \subset 
(\cW \cap \ker C)^\top =\cW^\top \oplus \im C^t \enspace .
\end{eqnarray*}
Therefore, $\cW^\top$ is $(A^t,C^t)$-controlled invariant. 
\end{proof}

The following duality theorem establishes a bijective correspondence between 
closed controlled invariant semimodules and closed conditioned invariant 
congruences. This is the basis of the algorithmic results which follow, 
since dealing with invariant semimodules is technically simpler than dealing 
with invariant congruences (because, in particular, the later objects show a 
``doubling'' of the dimension). It is worth mentioning that for systems 
with coefficients in a ring the duality between controlled and 
conditioned invariant modules does not hold in general~\cite{DLLL08}. 
For systems in infinite dimensions on a Hilbert space as well, 
closeness is instrumental for obtaining duality results~\cite{Cu86}. 
In this case other hypothesis are necessary concerning the domain of the operators and their boundedness. 

\begin{theorem}[Duality theorem]\label{TheoCondContr}
Let $\cV\subset (\rmax^ n)^2$ be a congruence. 
If $\cW$ is a closed congruence, then 
\[
\cW\in \sL(C,A,\cV)\iff \cW^\top\in \sM(A^t,C^t,\cV^\top) \enspace .
\]
\end{theorem}

\begin{proof}
If the congruence $\cW$ is $(C,A)$-conditioned invariant and closed, 
then, by Lemma~\ref{conditionedcontrolled}, 
$\cW^\top$ is $(A^t,C^t)$-controlled invariant. 
Moreover, if $\cW\supset \cV$, it follows that $\cW^\top\subset \cV^\top$. 
This shows the ``only if'' part of the theorem.

Conversely, if $\cX:=\cW^\top$ is $(A^t,C^t)$-controlled invariant, 
then, by Lemma~\ref{conditionedcontrolled}, 
$\cX^\bot$ is $(C,A)$-conditioned invariant.
Moreover, if $\cW$ is closed and $\cX\subset \cV^\top$, then
$\cW=(\cW^\top)^\bot=\cX^\bot\supset (\cV^\top)^\bot\supset \cV$,
which shows the ``if'' part of the theorem.
\end{proof}

As a consequence, we have. 

\begin{proposition}\label{controcondi}
Let $\cV\subset (\rmax^ n)^2$ be a congruence. If we define 
$\cK=\cV^\top$, then ${\cK^*(A^t,C^t)}^\bot$ is the minimal 
closed $(C,A)$-conditioned invariant congruence containing $\cV$. 
Therefore, $\cV_*(C,A)\subset {\cK^*(A^t,C^t)}^\bot$ and 
$\cV_*(C,A)= {\cK^*(A^t,C^t)}^\bot$ if $\cV_*(C,A)$ is closed.
\end{proposition}

\begin{proof}
By Lemma~\ref{conditionedcontrolled} we know that 
${\cK^*(A^t,C^t)}^\bot$ is $(C,A)$-conditioned invariant. 
Moreover, since $\cK^*(A^t,C^t)\subset \cK$, we have  
$\cV\subset (\cV^\top)^\bot =\cK^\bot \subset {\cK^*(A^t,C^t)}^\bot$. 
 
Let $\cW$ be a closed $(C,A)$-conditioned 
invariant congruence containing $\cV$. 
Then, by Theorem~\ref{TheoCondContr}, we have 
$\cW^\top\in \sM(A^t,C^t,\cK)$ and thus 
$\cW^\top \subset \cK^*(A^t,C^t)$. Therefore, 
${\cK^*(A^t,C^t)}^\bot \subset (\cW^\top)^\bot=\cW$.  
\end{proof}

Since in the previous proposition $\cK=\cV^\top$ is closed, 
we can apply Proposition~\ref{Kclosed} in order to compute $\cK^*(A^t,C^t)$.  
This means that in~\eqref{DefPhiX} and~\eqref{DefSequenceX} 
we have to take $\cK=\cV^\top$, $B=C^t$ and $A^t$ instead of $A$. 

\begin{example}
Consider again the matrices $A$ and $C$ and the congruence $\cV$ 
of Example~\ref{ExampleSeqW}. We have seen that in this case 
$\cV_*(C,A)$ is not closed. Taking $\cK=\cV^\top$, 
by Proposition~\ref{controcondi}, we know that 
${\cK^*(A^t,C^t)}^\bot$ is the minimal closed 
$(C,A)$-conditioned invariant congruence containing $\cV$.
Note that $\cV=\ker E$, where
\[
 E =
\begin{pmatrix}
\unit & \zero \cr
\unit & \unit \cr 
\end{pmatrix} \; ,
\]
so that $\cK=\im E^t$ is the semimodule considered 
in Example~\ref{ExampleSeqX}. Since $A=A^t$ and the matrix $B$ 
of Example~\ref{ExampleSeqX} is equal to $C^t$, 
$\cK^*(A^t,C^t)$ is the semimodule 
$\cK^*(A,B)=\set{x\in \rmax^2}{x_2=\zero}$ computed 
in Example~\ref{ExampleSeqX}. Therefore, we conclude that 
the minimal closed $(C,A)$-conditioned invariant congruence 
containing $\cV$ is 
${\cK^*(A^t,C^t)}^\bot = \set{(x,y)\in (\rmax^2)^2}{x_1=y_1}$.
\end{example}

We say that a congruence $\cW$ is cofinitely generated if 
$\cW=\ker E$ for some matrix $E$. Since a congruence $\cW$ on $\rmax^n$ is 
in particular a subsemimodule of $(\rmax^n)^2$, we say that 
$\cW$ is finitely generated if it is finitely generated as a semimodule, 
that is, if there exists a finite family 
$\left\{(x_i,y_i)\right\}_{i\in I}\subset (\rmax^n)^2$ 
such that $\cW$ is the set of elements of the form 
$\bigoplus_{i\in I}\lambda_i (x_i,y_i)$, with $\lambda_i\in \rmax$. 
We next show that the class of cofinitely generated congruences coincides with 
the class of finitely generated congruences. With this aim, 
we shall need the following lemma, which tells us that the 
solution sets of homogeneous max-plus linear systems of 
equations are finitely generated semimodules. 

\begin{lemma}[\cite{butkovicH,gaubert92a}]
If $F$ and $G$ are two rectangular matrices of the same dimension, 
then, $\set{z}{Fz=Gz}$ is a finitely generated semimodule.
\end{lemma}

\begin{lemma}\label{th-C}
A congruence $\cW$ is cofinitely generated if, and only
if, it is finitely generated as a semimodule.
\end{lemma}

\begin{proof}
If $\cW=\ker E$, then, it is finitely generated as a semimodule,
because $\cW=\set{(x,y)}{Ex=Ey}$
is the solution set of an homogeneous max-plus linear system of equations,
which is finitely generated as a semimodule by the previous lemma.

Conversely, if $\cW$ is finitely generated as a semimodule,
then it is closed, and so $\cW=(\cW^\top)^{\bot}$ by Theorem~\ref{SepTheo}. 
Using again the previous lemma, 
we deduce that $\cW^\top$ is a finitely generated semimodule, 
so $\cW^\top =\im G$ for some matrix $G$.
It follows that $\cW=(\im G)^\bot=\ker G^t$ is cofinitely generated.
\end{proof}

In practice, the objects of interest are usually finitely generated 
congruences and semimodules. This is why, in the sequel, 
we focus our attention to them and consider the following sets  
\[
\sLg(C,A,\cV)=
\set{\cW\in \sL(C,A,\cV)}{\cW=\ker D \; \text{\rm for some matrix } D}
\]
and 
\[
\sMg(A,B,\cK)=
\set{\cX\in \sM(A,B,\cK)}{\cX =\im D \; \text{\rm for some matrix } D} \; . 
\]
The following theorem relates $\sLg(C,A,\cV)$ with $\sMg(A^t,C^t,\cV^\top)$. 

\begin{theorem}\label{FGDuality}
Let $\cV\subset (\rmax^ n)^2$ be a congruence. If we define 
$\cK=\cV^\top$, then $\sLg(C,A,\cV)$ 
admits a minimal element $\cVg(C,A)$ if, and only if, 
$\sMg(A^t,C^t,\cK)$ admits a maximal element 
$\cKg(A^t,C^t)$. Moreover, when these elements exist, they 
satisfy $\cVg(C,A)= \cKg(A^t,C^t)^\bot$.
\end{theorem}

\begin{proof}
In the first place, note that 
\[
\ker D_1\subset \ker D_2 \iff \im D_2^t \subset \im D_1^t\; , 
\] 
since 
$\ker D_1\subset \ker D_2 \implies (\ker D_2)^\top\subset (\ker D_1)^\top 
\implies \im D_2^t \subset \im D_1^t$ and 
$\im D_2^t \subset \im D_1^t \implies 
(\im D_1^t)^\bot \subset (\im D_2^t)^\bot \implies \ker D_1\subset \ker D_2$ 
by Lemma~\ref{LemmaFG}. In addition, 
by Theorem~\ref{TheoCondContr} we know that 
$\cW=\ker D\in \sLg(C,A,\cV)$ if, and only if, 
$\cX=\im D^t\in \sMg(A^t,C^t,\cK)$. 
Therefore, $\sLg(C,A,\cV)$ admits a minimal element 
$\cVg(C,A)$ if, and only if, $\sMg(A^t,C^t,\cK)$ admits a maximal 
element $\cKg(A^t,C^t)$. Moreover, if $\cVg(C,A)$ exists and 
$\cVg(C,A)=\ker D$, then $\cKg(A^t,C^t)=\im D^t$ and thus 
$\cVg(C,A)=\cKg(A^t,C^t)^\bot$.
\end{proof}
   
We next give a condition which ensures the existence of $\cVg(C,A)$. 
With this aim, it is convenient to restrict ourselves to the 
subsemiring $\zmax=(\Zm,\max,+)$ of $\rmax$ 
and introduce the notion of volume of a subsemimodule of $\zmax^n$.

\begin{definition}\label{defvolumen}
Let $\cK\subset \zmax^n$ be a semimodule. We call {\em volume} of $\cK$, 
and we represent it with $\vol(\cK)$, the cardinality of the set 
$\set{z\in \cK}{z_1\oplus \cdots \oplus z_n=0}$. Moreover, 
if $E\in \zmax^{n\times p}$, we represent with $\vol(E)$ 
the volume of the semimodule $\cK=\im E$, that is, 
$\vol(E)=\vol(\im E)$. 
\end{definition}

Note that a semimodule $\cK\subset \zmax^n$ with finite volume is 
necessarily finitely generated because clearly  
$\cK = \smallvect \set{z\in \cK}{z_1\oplus \cdots \oplus z_n=0}$. 
We shall need the following properties. 

\begin{lemma}[\cite{katz07}]\label{lemapropvol}
Let $E\in \zmax^{n\times p}$ be a matrix 
and $\cZ,\cY\subset  \zmax^n$ be semimodules. Then, we have
\begin{enumerate}
\item[(i)]\label{pv1} 
$\cY \subset \cZ \Rightarrow \vol(\cY) \leq \vol(\cZ) \;$,
\item[(ii)]\label{pv2} if $\vol(\cY)<\infty$, 
then $\cY \varsubsetneq  \cZ \Rightarrow \vol(\cY) <\vol(\cZ) \;$,
\item[(iii)]\label{pv3} $\vol(E) =\vol(E^t)\; $.
\end{enumerate}
\end{lemma}

\begin{proposition}\label{PropFiniteVolume}
Let $A\in \zmax^{n\times n}$, $C\in \zmax^{q\times n}$ 
and $\cV=\ker E$, where $E\in \zmax^{p\times n}$ is a matrix 
with finite volume. Then, $\cVg(C,A)$ exists. Moreover,  
if we define the sequence $\left\{\cX_k\right\}_{k\in \N}$ 
by~\eqref{DefSequenceX}, with $\cK=\cV^\top$, $B=C^t$ and 
$A^t$ instead of $A$ in~\eqref{DefPhiX}, then 
$\cVg(C,A)={\cX_r}^\bot$ for some $r\leq \vol (E)+1$.
\end{proposition}

\begin{proof} Firstly, note that by Property~(iii) 
of Lemma~\ref{lemapropvol} we have
$\vol(\cX_1)=\vol(\cK)= \vol(E^t)=\vol(E)<\infty $. Then, 
since $\left\{\cX_k\right\}_{k\in \N}$ is a decreasing 
sequence of semimodules, from Property~(i) 
of Lemma~\ref{lemapropvol} we deduce that 
$\left\{\vol(\cX_k)\right\}_{k\in \N}$ 
is a decreasing sequence of non-negative integers. 
Therefore, there exists $r\leq \vol(\cX_1)+1=\vol(E)+1$ such that 
$\vol(\cX_{r+1})=\vol(\cX_r)$, so by Property~(ii) 
of Lemma~\ref{lemapropvol}, we have $\cX_{r+1}=\cX_r$. 
It follows that $\cX_k=\cX_r$ for all $k\geq r$ and thus 
$\cK^*(A^t,C^t)=\cK^\infty=\cX_r$ by Proposition~\ref{Kclosed}. 
Finally, since $\cX_r$ has finite volume, 
and so it is finitely generated, 
we conclude that $\cKg(A^t,C^t)=\cK^*(A^t,C^t)=\cX_r$ and then, 
by Theorem~\ref{FGDuality} we have $\cVg(C,A)= \cX_r^\bot$. 
\end{proof}

\begin{remark}
Note that in the previous proof we in particular showed that when $\cK$ 
has finite volume, the sequence $\left\{\cX_k\right\}_{k\in \N}$ defined 
in~\eqref{DefSequenceX} converges in at most $\vol(\cK)+1$ steps to 
$\cK^*(A,B)$, for any pair of matrices $A$ and $B$.
\end{remark}

For sufficient conditions for $\cK=\im E$ to have finite volume,  
and a bound for $\vol (E)$ in terms of the additive version of 
Hilbert's projective metric when $E$ only has finite entries, 
we refer the reader to~\cite{katz07}.  

To end this section, observe that Proposition~\ref{CondInvInterp} raises 
the question of constructing a dynamic observer for system~\eqref{e-fond},
allowing us to compute $[x(m)]_{\cW}$ as a function of 
$[x(0)]_{\cW}$ and $y(0),\ldots, y(m-1)$. To do so, 
we shall assume that the congruence $\cW$ is cofinitely generated, 
so that $\cW=\ker F$ for some matrix $F$. Then, 
we must compute $Fx(m)$ in terms of $Fx(0)$ and $y(0),\ldots,y(m-1)$.

\begin{theorem}[Dynamic observer]\label{TheoObserver}
Assume that the minimal $(C,A)$-conditioned
invariant congruence $\cW$ containing $\cV$
is cofinitely generated, so that $\cW=\ker F$ for some matrix $F$.
Then, there exist two matrices $U$ and $V$ such that
\begin{align}\label{e-transpose}
FA=UF\oplus VC \enspace ,
\end{align}
and for any choice of these matrices, 
if we define $z(k):=F x(k)$, we have
\begin{align}\label{e-do}
z(k+1)=Uz(k)\oplus Vy(k) \; \makebox{ for }\; k\geq 0 \; ,
\end{align}
where $\left\{ x(k)\right\}_{k\geq 0}$ 
is any trajectory of system~\eqref{e-fond} and 
$\left\{ y(k)\right\}_{k\geq 0}$ 
is the corresponding output trajectory.
\end{theorem}

\begin{proof}
By Theorem~\ref{TheoCondContr} we know that $\cW^\top$ is 
$(A^t,C^t)$-controlled invariant and thus 
\[
A^t(\cW^\top)\subset \cW^\top\oplus \im C^t \; .
\]
Since $\cW^\top =\im F^t$ by Lemma~\ref{LemmaFG}, we deduce that
\[
A^tF^t=F^tU^t\oplus C^tV^t
\]
for some matrices $U$ and $V$. After transposing, 
we obtain~\eqref{e-transpose}. 
Since $\cV \subset \ker F$,
we have $z(k+1)=Fx(k+1)=FAx(k)$,
and using~\eqref{e-transpose}, we get
\[
F A x(k)=U F x(k)\oplus V C x(k)= U z(k)\oplus V y(k)\; , 
\]
which shows~\eqref{e-do}.
\end{proof}

\begin{remark}
Suppose that, given a matrix $G$, we want to 
construct an observer for reconstructing, 
from the observation and initial condition, 
the linear functional of the state of system~\eqref{e-fond}
\[
w(k)=Gx(k) \; , 
\]
where $k\geq 0$. With this aim, assume that the minimal $(C,A)$-conditioned
invariant congruence containing $\cV$
is contained in $\ker G$ and cofinitely generated, so that 
$\cV_*(C,A)=\ker F\subset \ker G$ for some matrix $F$. 
Then, if we define $z(k):=Fx(k)$ like in Theorem~\ref{TheoObserver}, 
we have $w(k)=G (-F^t) z(k)$ for all $k\geq 0$, 
where the product by $-F^t$ is performed in the semiring 
$(\R\cup\{-\infty,+\infty\},\min ,+)$ 
with the convention $(+\infty)+a=+\infty$ 
for all $a\in\R\cup\{-\infty,+\infty\}$. 
As a matter of fact, from the equality $F (-F^t) F=F$ 
(see for example~\cite{bcoq,BlythJan72}), 
it follows that  
\[
F (-F^t) z(k) = F (-F^t) F x(k) = F x(k) \; ,
\] 
and since $\ker F \subset \ker G$, 
we get $G (-F^t) z(k) = G x(k) =w(k)$. 

Therefore, by Theorem~\ref{TheoObserver}, we conclude  
that it is possible to effectively reconstruct the linear functional of 
the state $w(k)=Gx(k)$ from the observation and initial state 
if the minimal $(C,A)$-conditioned invariant congruence 
containing $\cV$ is contained in $\ker G$ and cofinitely generated.  
\end{remark}

\section{Application to a manufacturing system}\label{SApplication}

Dynamical systems of the form~\eqref{e-fond} can be used, 
for instance, to model max-plus linear dynamical systems of the form 
\[
x(k+1)=\bar{A} x(k)\; , 
\]
where some entries of $\bar{A}$ are unknown but belong to certain intervals, 
at the price of adding new variables. To see this, 
assume for example that in the max-plus linear dynamical system 
\begin{equation}\label{persystem}
\begin{pmatrix}
x_1(k+1) \cr 
x_2(k+1)\cr
\end{pmatrix}= 
\begin{pmatrix}
a_{11} & a_{12} \cr
a_{21}(k) & a_{22} \cr 
\end{pmatrix}
\begin{pmatrix}
x_1(k) \cr 
x_2(k)\cr
\end{pmatrix} \; , 
\end{equation}
 $a_{21}(k)$ can take any value in the interval $[a,b]$ for each $k\in \N$, 
but $a_{11}$, $a_{12}$ and $a_{22}$ are fixed. 
Consider the ``extended state'' vector  
$x(k)=\left( x_1(k), x_2(k), x_3(k), x_4(k) \right)^t$ 
and define 
\[ 
A= 
\begin{pmatrix}
a_{11} & a_{12} & \zero & \zero \cr
\zero & a_{22} & a & b \cr 
a_{11} & a_{12} & \zero & \zero \cr
a_{11} & a_{12} & \zero & \zero \cr
\end{pmatrix}
\; \mbox{ and } \;   
E= 
\begin{pmatrix}
0 & \zero & \zero & \zero \cr 
\zero & 0 & \zero & \zero \cr
\zero & \zero & 0 & 0 \cr
\end{pmatrix} \; .
\]
Then, as 
\[ 
Ax(k)= 
\begin{pmatrix}
a_{11}x_1(k)\oplus a_{12}x_2(k) \cr 
a_{22}x_2(k)\oplus ax_3(k)\oplus bx_4(k) \cr
a_{11}x_1(k)\oplus a_{12}x_2(k) \cr
a_{11}x_1(k)\oplus a_{12}x_2(k) \cr
\end{pmatrix} 
\; , 
\]
we have 
\[
Ex(k+1) = EAx(k)\iff \left\{ \begin{array}{l}
x_1(k+1) = a_{11}x_1(k)\oplus a_{12}x_2(k) \\
x_2(k+1) = a_{22}x_2(k)\oplus ax_3(k)\oplus bx_4(k) \\ 
x_3(k+1)\oplus x_4(k+1)= a_{11}x_1(k)\oplus a_{12}x_2(k)
\end{array}\right. \; .
\] 
Since 
\[
x_3(k+1)\oplus x_4(k+1)=a_{11}x_1(k)\oplus a_{12} x_2(k)= x_1(k+1)
\] 
it follows that 
\[ 
a x_3(k+1)\oplus b x_4(k+1)=a_{21}(k+1) x_1(k+1)
\] 
for some $a_{21}(k+1)\in [a,b]$. Thus, 
if we assume that the initial state $x(0)$ satisfies 
the condition $x_3(0)\oplus x_4(0)= x_1(0)$, for all $k\geq 0$ 
we have 
\[
x_2(k+1) = a_{22}x_2(k)\oplus ax_3(k)\oplus bx_4(k) 
= a_{22}x_2(k)\oplus a_{21}(k) x_1(k) 
\]
for some $a_{21}(k)\in [a,b]$. Therefore, the first two entries 
of the state vector of the dynamical system 
\[
Ex(k+1) = EAx(k)
\]
describe the evolution of system~\eqref{persystem}, 
in the sense that they are equal to the state vector of 
system~\eqref{persystem} corresponding to some choice of 
$a_{21}(k)$ in $[a,b]$ for each $k\in \N$, and vice versa. 

This idea can be generalized to the case of more 
than one uncertain holding time by adding two auxiliary 
variables for each of them. In particular, 
this method can be used to model manufacturing systems 
and transportation networks in which some (processing or traveling) 
times are unknown but bounded. 
When applying it, in order to satisfy the 
previous condition on the initial state of the extended state vector, 
for simplicity we will assume that $x_i(0)=\unit$ for all $i$.  

\begin{figure}
\begin{center}
\begin{picture}(0,0)%
\includegraphics{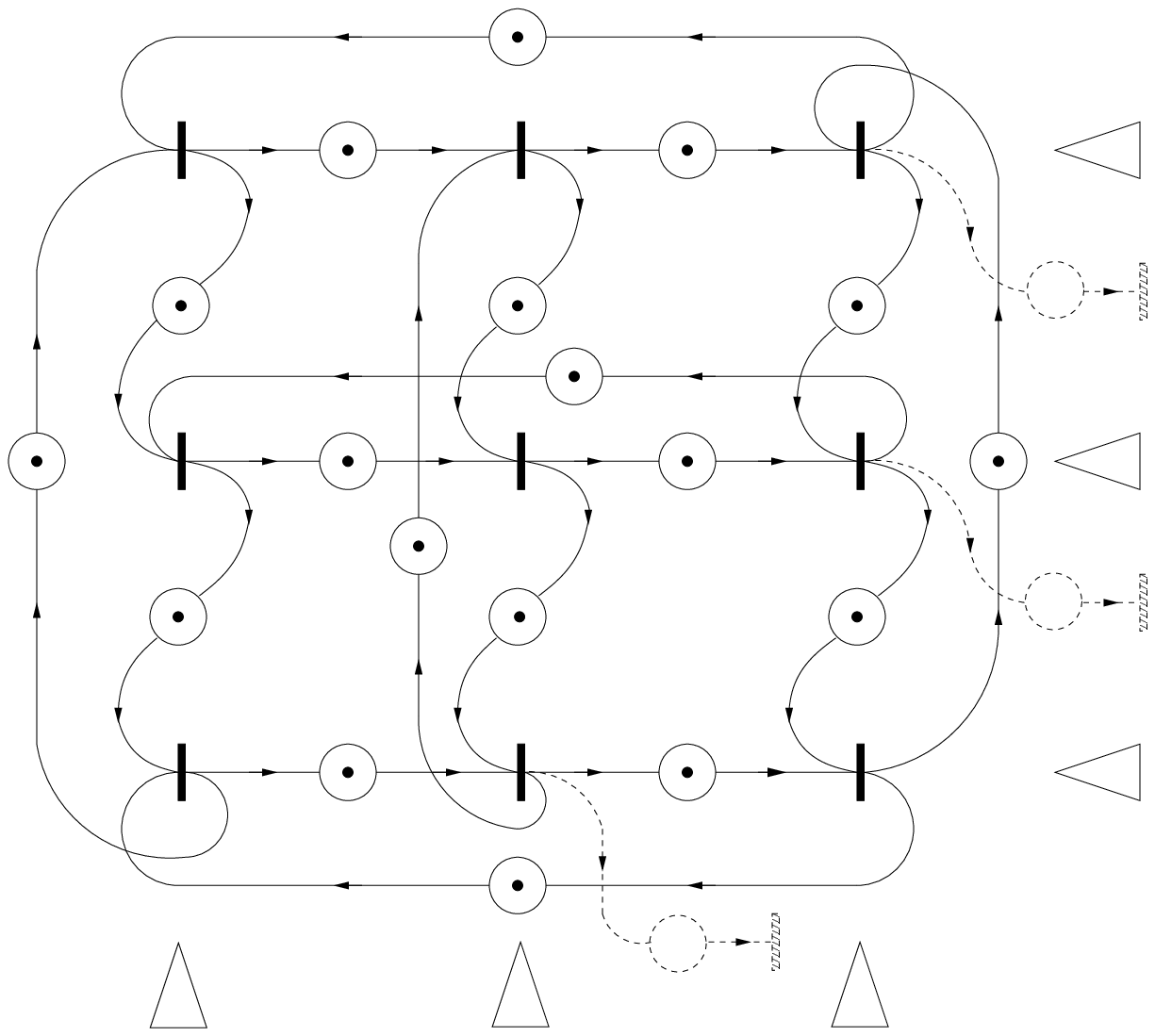}%
\end{picture}%
\setlength{\unitlength}{1895sp}%
\begingroup\makeatletter\ifx\SetFigFont\undefined%
\gdef\SetFigFont#1#2#3#4#5{%
  \reset@font\fontsize{#1}{#2pt}%
  \fontfamily{#3}\fontseries{#4}\fontshape{#5}%
  \selectfont}%
\fi\endgroup%
\begin{picture}(12765,11052)(376,-10282)
\put(5326,-5236){\makebox(0,0)[lb]{\smash{{\SetFigFont{8}{9.6}{\rmdefault}{\mddefault}{\updefault}{\color[rgb]{0,0,0}$3$}%
}}}}
\put(6376,-5986){\makebox(0,0)[lb]{\smash{{\SetFigFont{8}{9.6}{\rmdefault}{\mddefault}{\updefault}{\color[rgb]{0,0,0}$3$}%
}}}}
\put(2776,-5986){\makebox(0,0)[lb]{\smash{{\SetFigFont{8}{9.6}{\rmdefault}{\mddefault}{\updefault}{\color[rgb]{0,0,0}$4$}%
}}}}
\put(5926,-3811){\makebox(0,0)[lb]{\smash{{\SetFigFont{8}{9.6}{\rmdefault}{\mddefault}{\updefault}{\color[rgb]{0,0,0}$x_5$}%
}}}}
\put(9526,-3811){\makebox(0,0)[lb]{\smash{{\SetFigFont{8}{9.6}{\rmdefault}{\mddefault}{\updefault}{\color[rgb]{0,0,0}$x_6$}%
}}}}
\put(9526,-511){\makebox(0,0)[lb]{\smash{{\SetFigFont{8}{9.6}{\rmdefault}{\mddefault}{\updefault}{\color[rgb]{0,0,0}$x_3$}%
}}}}
\put(7726,-1486){\makebox(0,0)[lb]{\smash{{\SetFigFont{8}{9.6}{\rmdefault}{\mddefault}{\updefault}{\color[rgb]{0,0,0}$5$}%
}}}}
\put(2326,-3811){\makebox(0,0)[lb]{\smash{{\SetFigFont{8}{9.6}{\rmdefault}{\mddefault}{\updefault}{\color[rgb]{0,0,0}$x_4$}%
}}}}
\put(5926,-7111){\makebox(0,0)[lb]{\smash{{\SetFigFont{8}{9.6}{\rmdefault}{\mddefault}{\updefault}{\color[rgb]{0,0,0}$x_8$}%
}}}}
\put(2326,-7111){\makebox(0,0)[lb]{\smash{{\SetFigFont{8}{9.6}{\rmdefault}{\mddefault}{\updefault}{\color[rgb]{0,0,0}$x_7$}%
}}}}
\put(9526,-7111){\makebox(0,0)[lb]{\smash{{\SetFigFont{8}{9.6}{\rmdefault}{\mddefault}{\updefault}{\color[rgb]{0,0,0}$x_9$}%
}}}}
\put(4126,-8086){\makebox(0,0)[lb]{\smash{{\SetFigFont{8}{9.6}{\rmdefault}{\mddefault}{\updefault}{\color[rgb]{0,0,0}$5$}%
}}}}
\put(5926,-9286){\makebox(0,0)[lb]{\smash{{\SetFigFont{8}{9.6}{\rmdefault}{\mddefault}{\updefault}{\color[rgb]{0,0,0}$3$}%
}}}}
\put(7726,-4786){\makebox(0,0)[lb]{\smash{{\SetFigFont{8}{9.6}{\rmdefault}{\mddefault}{\updefault}{\color[rgb]{0,0,0}$4$}%
}}}}
\put(9976,-5986){\makebox(0,0)[lb]{\smash{{\SetFigFont{8}{9.6}{\rmdefault}{\mddefault}{\updefault}{\color[rgb]{0,0,0}$2$}%
}}}}
\put(9976,-2686){\makebox(0,0)[lb]{\smash{{\SetFigFont{8}{9.6}{\rmdefault}{\mddefault}{\updefault}{\color[rgb]{0,0,0}$5$}%
}}}}
\put(6376,-2686){\makebox(0,0)[lb]{\smash{{\SetFigFont{8}{9.6}{\rmdefault}{\mddefault}{\updefault}{\color[rgb]{0,0,0}$\left[3,5\right]$}%
}}}}
\put(2851,-2686){\makebox(0,0)[lb]{\smash{{\SetFigFont{8}{9.6}{\rmdefault}{\mddefault}{\updefault}{\color[rgb]{0,0,0}$4$}%
}}}}
\put(6526,-3886){\makebox(0,0)[lb]{\smash{{\SetFigFont{8}{9.6}{\rmdefault}{\mddefault}{\updefault}{\color[rgb]{0,0,0}$3$}%
}}}}
\put(2326,-511){\makebox(0,0)[lb]{\smash{{\SetFigFont{8}{9.6}{\rmdefault}{\mddefault}{\updefault}{\color[rgb]{0,0,0}$x_1$}%
}}}}
\put(12451,-2011){\makebox(0,0)[lb]{\smash{{\SetFigFont{8}{9.6}{\rmdefault}{\mddefault}{\updefault}{\color[rgb]{0,0,0}$y_1$}%
}}}}
\put(7651,-9886){\makebox(0,0)[lb]{\smash{{\SetFigFont{8}{9.6}{\rmdefault}{\mddefault}{\updefault}{\color[rgb]{0,0,0}$0$}%
}}}}
\put(11626,-6286){\makebox(0,0)[lb]{\smash{{\SetFigFont{8}{9.6}{\rmdefault}{\mddefault}{\updefault}{\color[rgb]{0,0,0}$0$}%
}}}}
\put(11626,-2986){\makebox(0,0)[lb]{\smash{{\SetFigFont{8}{9.6}{\rmdefault}{\mddefault}{\updefault}{\color[rgb]{0,0,0}$0$}%
}}}}
\put(8626,-9886){\makebox(0,0)[lb]{\smash{{\SetFigFont{8}{9.6}{\rmdefault}{\mddefault}{\updefault}{\color[rgb]{0,0,0}$y_3$}%
}}}}
\put(12526,-5311){\makebox(0,0)[lb]{\smash{{\SetFigFont{8}{9.6}{\rmdefault}{\mddefault}{\updefault}{\color[rgb]{0,0,0}$y_2$}%
}}}}
\put(12151,-7636){\makebox(0,0)[lb]{\smash{{\SetFigFont{8}{9.6}{\rmdefault}{\mddefault}{\updefault}{\color[rgb]{0,0,0}$M_3$}%
}}}}
\put(12151,-4336){\makebox(0,0)[lb]{\smash{{\SetFigFont{8}{9.6}{\rmdefault}{\mddefault}{\updefault}{\color[rgb]{0,0,0}$M_2$}%
}}}}
\put(12151,-1036){\makebox(0,0)[lb]{\smash{{\SetFigFont{8}{9.6}{\rmdefault}{\mddefault}{\updefault}{\color[rgb]{0,0,0}$M_1$}%
}}}}
\put(10576,-4336){\makebox(0,0)[lb]{\smash{{\SetFigFont{8}{9.6}{\rmdefault}{\mddefault}{\updefault}{\color[rgb]{0,0,0}$1$}%
}}}}
\put(376,-4336){\makebox(0,0)[lb]{\smash{{\SetFigFont{8}{9.6}{\rmdefault}{\mddefault}{\updefault}{\color[rgb]{0,0,0}$2$}%
}}}}
\put(5926,614){\makebox(0,0)[lb]{\smash{{\SetFigFont{8}{9.6}{\rmdefault}{\mddefault}{\updefault}{\color[rgb]{0,0,0}$4$}%
}}}}
\put(5926,-511){\makebox(0,0)[lb]{\smash{{\SetFigFont{8}{9.6}{\rmdefault}{\mddefault}{\updefault}{\color[rgb]{0,0,0}$x_2$}%
}}}}
\put(7726,-8086){\makebox(0,0)[lb]{\smash{{\SetFigFont{8}{9.6}{\rmdefault}{\mddefault}{\updefault}{\color[rgb]{0,0,0}$4$}%
}}}}
\put(2251,-10111){\makebox(0,0)[lb]{\smash{{\SetFigFont{8}{9.6}{\rmdefault}{\mddefault}{\updefault}{\color[rgb]{0,0,0}$P_1$}%
}}}}
\put(5851,-10111){\makebox(0,0)[lb]{\smash{{\SetFigFont{8}{9.6}{\rmdefault}{\mddefault}{\updefault}{\color[rgb]{0,0,0}$P_2$}%
}}}}
\put(9451,-10111){\makebox(0,0)[lb]{\smash{{\SetFigFont{8}{9.6}{\rmdefault}{\mddefault}{\updefault}{\color[rgb]{0,0,0}$P_3$}%
}}}}
\put(3901,-4786){\makebox(0,0)[lb]{\smash{{\SetFigFont{8}{9.6}{\rmdefault}{\mddefault}{\updefault}{\color[rgb]{0,0,0}$\left[1,3\right]$}%
}}}}
\put(3901,-1486){\makebox(0,0)[lb]{\smash{{\SetFigFont{8}{9.6}{\rmdefault}{\mddefault}{\updefault}{\color[rgb]{0,0,0}$\left[1,7\right]$}%
}}}}
\end{picture}%
\caption{A manufacturing system (flow-shop)}
\label{figure1}
\end{center}
\end{figure} 

As an example, consider the Timed Event Graph of Figure~\ref{figure1}.
This figure represents a manufacturing system (flow-shop) composed 
of three machines, denoted by $M_1$, $M_2$ and $M_3$, 
which is supposed to produce three kinds of parts, 
denoted by $P_1$, $P_2$ and $P_3$  (we refer the reader to~\cite{bcoq} 
for background on the modeling of Timed Event Graphs using max-plus algebra). 
We assume that each machine processes each part exactly once, 
that all parts follow the same sequence of machines: 
$M_1$, $M_2$ and finally $M_3$, and that the sequencing 
of part types on each machine is the same: $P_1$, $P_2$ and finally $P_3$. 
Parts are carried on pallets form one machine to the next one. 
When a part has been processed by the three machines, 
it is removed from the pallet, 
which returns to the staring point for a new part.

In Figure~\ref{figure1}, each of the nine transitions corresponds 
to a combination of a machine and a part type. For instance, the 
transition labeled $x_6$ corresponds to the combination of 
machine $M_2$ processing part $P_3$. To each transition 
corresponds a variable $x_i(k)$ which denotes the earliest time at 
which the transition can be fired for $k$-th time, that is, 
the earliest time at which a specific machine can start 
processing a specific part type for $k$-th time. 

Places between transitions express the precedence constrains 
between operations due to the sequencing of operations on the machines. 
For instance, $x_6$ depends on $x_3$, which corresponds to $M_1$ 
processing $P_3$, and on $x_5$, which corresponds to $M_2$ processing $P_2$. 
The holding time assigned to each place is determined, 
for instance, as a function of some or all of the following variables: 
the processing time of machines on parts, 
the transportation time between machines and 
the set up time on machines when switching from 
one part type to another. These times are given in Figure~\ref{figure1}. 
Note that the times $x_1\rightarrow x_2$, $x_2\rightarrow x_5$ 
and $x_4\rightarrow x_5$ are not fixed but are 
assumed to belong to the intervals $\left[ 1,7 \right]$, 
$\left[ 3,5 \right]$ and $\left[ 1,3 \right]$ respectively. 
This variation could be due, for instance, to possible breakdowns. 
All the other times are supposed to be fixed. 

For simplicity, we assume that in the initial state 
there is a token in each place. This physically means that 
each machine can process at most three parts at the same time, 
and that there are three pallets carrying each part type.

The evolution of this flow-shop can be described by a max-plus linear 
dynamical system of the form $x(k+1)=\bar{A} x(k)$, 
where $x(k)\in \rmax^9$ is the vector of $k$-th 
firing times of the nine transitions and $\bar{A}_{ij}$ is the holding time 
of the place in the arc that goes from $x_j$ to $x_i$ ($\bar{A}_{ij}=\zero$ 
if there is no such an arc, see~\cite{bcoq} for details). 
Due to the presence of uncertain holding times, 
three entries of $\bar{A}$ may vary with $k$, 
so we next use the method described above to model this system.  
After adding six auxiliary variables $x_i$, $i=10,\dots ,15$ 
(two for each uncertain holding time), the evolution of the 
flow-shop of Figure~\ref{figure1} can be described by the 
following dynamical system
\begin{equation}\label{EqEvSystem}
Ex(k+1) = EAx(k) \; , 
\end{equation}
where 
\[  
E= 
\left(
\begin{array}{ccccccccccccccc}
0 & \zero & \zero & \zero & \zero & \zero & \zero & \zero & \zero & \zero & \zero & \zero & \zero & \zero & \zero \\ 
\zero & 0 & \zero & \zero & \zero & \zero & \zero & \zero & \zero & \zero & \zero & \zero & \zero & \zero & \zero \\
\zero & \zero & 0 & \zero & \zero & \zero & \zero & \zero & \zero & \zero & \zero & \zero & \zero & \zero & \zero \\
\zero & \zero & \zero & 0 & \zero & \zero & \zero & \zero & \zero & \zero & \zero & \zero & \zero & \zero & \zero \\
\zero & \zero & \zero & \zero & 0  & \zero & \zero & \zero & \zero & \zero & \zero & \zero & \zero & \zero & \zero \\
\zero & \zero & \zero & \zero & \zero & 0 & \zero & \zero & \zero & \zero & \zero & \zero & \zero & \zero & \zero \\
\zero & \zero & \zero & \zero & \zero & \zero & 0 & \zero & \zero & \zero & \zero & \zero & \zero & \zero & \zero \\
\zero & \zero & \zero & \zero & \zero & \zero & \zero & 0 & \zero & \zero & \zero & \zero & \zero & \zero & \zero \\
\zero & \zero & \zero & \zero & \zero & \zero & \zero & \zero & 0 & \zero & \zero & \zero & \zero & \zero & \zero \\
\zero & \zero & \zero & \zero & \zero & \zero & \zero & \zero & \zero & 0 & 0 & \zero & \zero & \zero & \zero \\
\zero & \zero & \zero & \zero & \zero & \zero & \zero & \zero & \zero & \zero & \zero & 0 & 0 & \zero & \zero \\
\zero & \zero & \zero & \zero & \zero & \zero & \zero & \zero & \zero & \zero & \zero  & \zero & \zero & 0 & 0
\end{array}
\right)
\]
and 
\[ 
A= 
\left(
\begin{array}{ccccccccccccccc}
\zero & \zero & 4 & \zero & \zero & \zero & 2 & \zero & \zero & \zero & \zero & \zero & \zero & \zero & \zero \\
\zero & \zero & \zero & \zero & \zero & \zero & \zero & 3 & \zero & 1 & 7 & \zero & \zero & \zero & \zero \\ 
\zero & 5 & \zero & \zero & \zero & \zero & \zero & \zero & 1 & \zero & \zero & \zero & \zero & \zero & \zero \\ 
4 & \zero & \zero & \zero & \zero & 3 & \zero & \zero & \zero & \zero & \zero & \zero & \zero & \zero & \zero \\
\zero & \zero & \zero & \zero & \zero & \zero & \zero & \zero & \zero & \zero & \zero & 3 & 5 & 1 & 3 \\ 
\zero & \zero & 5 & \zero & 4 & \zero & \zero & \zero & \zero & \zero & \zero & \zero & \zero & \zero & \zero \\ 
\zero & \zero & \zero & 4 & \zero & \zero & \zero & \zero & 3 & \zero & \zero & \zero & \zero & \zero & \zero \\
\zero & \zero & \zero & \zero & 3 & \zero & 5 &\zero & \zero & \zero & \zero & \zero & \zero & \zero & \zero \\
\zero & \zero & \zero & \zero & \zero & 2 & \zero & 4 & \zero & \zero & \zero & \zero & \zero & \zero & \zero \\
\zero & \zero & 4 & \zero & \zero & \zero & 2 & \zero & \zero & \zero & \zero & \zero & \zero & \zero & \zero \\
\zero & \zero & 4 & \zero & \zero & \zero & 2 & \zero & \zero & \zero & \zero & \zero & \zero & \zero & \zero \\
\zero & \zero & \zero & \zero & \zero & \zero & \zero & 3 & \zero & 1 & 7 & \zero & \zero & \zero & \zero \\ 
\zero & \zero & \zero & \zero & \zero & \zero & \zero & 3 & \zero & 1 & 7 & \zero & \zero & \zero & \zero \\ 
4 & \zero & \zero & \zero & \zero & 3 & \zero & \zero & \zero & \zero & \zero & \zero & \zero & \zero & \zero \\
4 & \zero & \zero & \zero & \zero & 3 & \zero & \zero & \zero & \zero & \zero & \zero & \zero & \zero & \zero 
\end{array}
\right)  \; .
\]

Now, assume that we observe the firing times of transitions $x_3$, $x_6$ and 
$x_8$, that is, we define $y(k)=Cx(k)$, where 
\[  
C= 
\left(
\begin{array}{ccccccccccccccc}
\zero & \zero & 0 & \zero & \zero & \zero & \zero & \zero & \zero & \zero & \zero & \zero & \zero & \zero & \zero \\
\zero & \zero & \zero & \zero & \zero & 0 & \zero & \zero & \zero & \zero & \zero & \zero & \zero & \zero & \zero \\
\zero & \zero & \zero & \zero & \zero & \zero & \zero & 0 & \zero & \zero & \zero & \zero & \zero & \zero & \zero 
\end{array}
\right) \; .
\]
Note that these times have a physical meaning. For instance, $x_8$ represents 
the time at which $M_3$ starts processing $P_2$.
 
Taking into account Theorem~\ref{FGDuality}, in order to determine if there 
exists a minimal cofinitely generated $(C,A)$-conditioned invariant congruence 
$\cVg(C,A)$ containing $\cV=\ker E$, in the first place we compute the maximal 
$(A^t,C^t)$-controlled invariant semimodule $\cK^*(A^t,C^t)$ contained 
in $\cK=\cV^\top=\im E^t$. With this aim, we apply Proposition~\ref{Kclosed} 
and compute the sequence of semimodules 
\begin{equation} 
\cX_1=\cK \; \makebox{ and }\; 
\cX_{k+1}=\phi(\cX_{k}) \; \makebox{ for }\; k\in \N \; , 
\end{equation}
where 
\begin{equation}
\phi (\cX)= \cK\cap (A^t)^{-1}(\cX \oplus \im C^t) \; . 
\end{equation} 
This can be done with the help of the max-plus toolbox of Scilab. 
To be more precise, this is performed expressing the intersection and inverse 
image of finitely generated semimodules as the solution sets of appropriate 
homogeneous max-plus linear systems of equations 
(see~\cite{gaubert98n}), 
and by solving these systems using the function {\em mpsolve} of this 
toolbox (see~\cite{AGG08} for a discussion 
of the complexity of the algorithm involved). 

In this way, we obtain $\cX_4=\cX_3\varsubsetneq \cX_2\varsubsetneq \cX_1$  
and thus $\cK^*(A^t,C^t)=\cX_3$, 
where $\cX_3$ is the semimodule generated by the rows of the following matrix
\[  
F= 
\left(
\begin{array}{ccccccccccccccc}
0 & \zero & \zero & \zero & \zero & \zero & \zero & \zero & \zero & \zero & \zero & \zero & \zero & \zero & \zero \\
\zero & \zero & \zero & 0 & \zero & \zero & \zero & \zero & \zero & \zero & \zero & \zero & \zero & \zero & \zero \\
\zero & \zero & \zero & \zero & \zero & \zero & 0 & \zero & \zero & \zero & \zero & \zero & \zero & \zero & \zero \\
\zero & \zero & \zero & \zero & \zero & \zero & \zero & \zero & 0 & \zero & \zero & \zero & \zero & \zero & \zero \\
\zero & \zero & \zero & \zero & \zero & \zero & \zero & \zero & \zero & 0 & 0 & \zero & \zero & \zero & \zero \\
\zero & \zero & \zero & \zero & \zero & \zero & \zero & \zero & \zero & \zero & \zero & \zero & \zero & 0 & 0 
\end{array}
\right) \; .
\]
Therefore, by Theorem~\ref{FGDuality} we conclude that $\cVg(A,C)=\ker F$. 

By Theorem~\ref{TheoObserver} 
we know that it is possible to reconstruct the functional of the state 
$z(k):=Fx(k)$ in terms of the initial condition $x(0)$ and the observations 
$y(0),\ldots ,y(m-1)$. More precisely, we have the following dynamic 
observer allowing us to compute $z(k)$: 
\begin{equation}\label{EqObserver}
\left\{ 
\begin{array}{l}
z(k+1)=Uz(k)\oplus Vy(k)  \\
z(0)=Fx(0) 
\end{array}
\right. 
\end{equation}
where 
\[
U= 
\left(
\begin{array}{cccccc}
\zero & \zero & 2 & \zero & \zero & \zero \\
4 & \zero & \zero & \zero & \zero & \zero \\
\zero & 4 & \zero & 3 & \zero & \zero \\
\zero & \zero & \zero & \zero & \zero & \zero \\
\zero & \zero & 2 & \zero & \zero & \zero \\
4 & \zero & \zero & \zero & \zero & \zero  
\end{array}
\right)
\;
\makebox{ and }
\;
V=
\left(
\begin{array}{ccc}
4 & \zero & \zero \\
\zero & 3 & \zero \\
\zero & \zero & \zero \\
\zero & 2 & 4 \\
4 & \zero & \zero \\
\zero & 3 & \zero 
\end{array}
\right) \; .
\]
The matrices $U$ and $V$ are obtained by 
solving the equation $FA=UF\oplus VC$ 
(this kind of one sided max-plus linear systems of equations 
can be efficiently solved with the help of residuation theory, 
see~\cite{BlythJan72,bcoq,gaubert98n,cuning}). 
Observe that, due to the form of $F$, 
this in particular means that we can determine $x_1(m)$, $x_4(m)$, 
$x_7(m)$ and $x_9(m)$ in terms of the initial condition and the observations.

\begin{figure}
\begin{center}
\begin{tabular}{cc}

\begin{picture}(0,0)%
\includegraphics{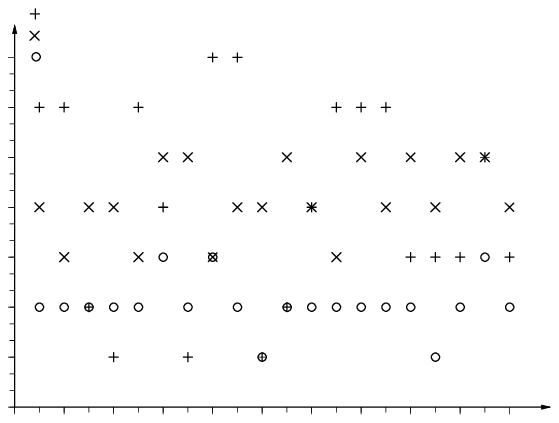}%
\end{picture}%
\setlength{\unitlength}{1302sp}%
\begingroup\makeatletter\ifx\SetFigFont\undefined%
\gdef\SetFigFont#1#2#3#4#5{%
  \reset@font\fontsize{#1}{#2pt}%
  \fontfamily{#3}\fontseries{#4}\fontshape{#5}%
  \selectfont}%
\fi\endgroup%
\begin{picture}(8278,6298)(901,-5459)
\put(1156,-5410){\makebox(0,0)[lb]{\smash{{\SetFigFont{5}{6.0}{\familydefault}{\mddefault}{\updefault}{\color[rgb]{0,0,0}$0$}%
}}}}
\put(1876,-5410){\makebox(0,0)[lb]{\smash{{\SetFigFont{5}{6.0}{\familydefault}{\mddefault}{\updefault}{\color[rgb]{0,0,0}$2$}%
}}}}
\put(2596,-5410){\makebox(0,0)[lb]{\smash{{\SetFigFont{5}{6.0}{\familydefault}{\mddefault}{\updefault}{\color[rgb]{0,0,0}$4$}%
}}}}
\put(3316,-5410){\makebox(0,0)[lb]{\smash{{\SetFigFont{5}{6.0}{\familydefault}{\mddefault}{\updefault}{\color[rgb]{0,0,0}$6$}%
}}}}
\put(4036,-5410){\makebox(0,0)[lb]{\smash{{\SetFigFont{5}{6.0}{\familydefault}{\mddefault}{\updefault}{\color[rgb]{0,0,0}$8$}%
}}}}
\put(4710,-5410){\makebox(0,0)[lb]{\smash{{\SetFigFont{5}{6.0}{\familydefault}{\mddefault}{\updefault}{\color[rgb]{0,0,0}$10$}%
}}}}
\put(5430,-5410){\makebox(0,0)[lb]{\smash{{\SetFigFont{5}{6.0}{\familydefault}{\mddefault}{\updefault}{\color[rgb]{0,0,0}$12$}%
}}}}
\put(6150,-5410){\makebox(0,0)[lb]{\smash{{\SetFigFont{5}{6.0}{\familydefault}{\mddefault}{\updefault}{\color[rgb]{0,0,0}$14$}%
}}}}
\put(6870,-5410){\makebox(0,0)[lb]{\smash{{\SetFigFont{5}{6.0}{\familydefault}{\mddefault}{\updefault}{\color[rgb]{0,0,0}$16$}%
}}}}
\put(7590,-5410){\makebox(0,0)[lb]{\smash{{\SetFigFont{5}{6.0}{\familydefault}{\mddefault}{\updefault}{\color[rgb]{0,0,0}$18$}%
}}}}
\put(8310,-5410){\makebox(0,0)[lb]{\smash{{\SetFigFont{5}{6.0}{\familydefault}{\mddefault}{\updefault}{\color[rgb]{0,0,0}$20$}%
}}}}
\put(901,-4466){\makebox(0,0)[lb]{\smash{{\SetFigFont{5}{6.0}{\familydefault}{\mddefault}{\updefault}{\color[rgb]{0,0,0}$1$}%
}}}}
\put(901,-3012){\makebox(0,0)[lb]{\smash{{\SetFigFont{5}{6.0}{\familydefault}{\mddefault}{\updefault}{\color[rgb]{0,0,0}$3$}%
}}}}
\put(901,-2286){\makebox(0,0)[lb]{\smash{{\SetFigFont{5}{6.0}{\familydefault}{\mddefault}{\updefault}{\color[rgb]{0,0,0}$4$}%
}}}}
\put(901,-1559){\makebox(0,0)[lb]{\smash{{\SetFigFont{5}{6.0}{\familydefault}{\mddefault}{\updefault}{\color[rgb]{0,0,0}$5$}%
}}}}
\put(901,-832){\makebox(0,0)[lb]{\smash{{\SetFigFont{5}{6.0}{\familydefault}{\mddefault}{\updefault}{\color[rgb]{0,0,0}$6$}%
}}}}
\put(901,-105){\makebox(0,0)[lb]{\smash{{\SetFigFont{5}{6.0}{\familydefault}{\mddefault}{\updefault}{\color[rgb]{0,0,0}$7$}%
}}}}
\put(901,-5161){\makebox(0,0)[lb]{\smash{{\SetFigFont{5}{6.0}{\familydefault}{\mddefault}{\updefault}{\color[rgb]{0,0,0}$0$}%
}}}}
\put(901,-3736){\makebox(0,0)[lb]{\smash{{\SetFigFont{5}{6.0}{\familydefault}{\mddefault}{\updefault}{\color[rgb]{0,0,0}$2$}%
}}}}
\put(8626,-4936){\makebox(0,0)[lb]{\smash{{\SetFigFont{8}{9.6}{\familydefault}{\mddefault}{\updefault}{\color[rgb]{0,0,0}$k$}%
}}}}
\put(1651,-61){\makebox(0,0)[lb]{\smash{{\SetFigFont{8}{9.6}{\familydefault}{\mddefault}{\updefault}{\color[rgb]{0,0,0}$x_4\rightarrow x_5$}%
}}}}
\put(1651,239){\makebox(0,0)[lb]{\smash{{\SetFigFont{8}{9.6}{\familydefault}{\mddefault}{\updefault}{\color[rgb]{0,0,0}$x_2\rightarrow x_5$}%
}}}}
\put(1651,539){\makebox(0,0)[lb]{\smash{{\SetFigFont{8}{9.6}{\familydefault}{\mddefault}{\updefault}{\color[rgb]{0,0,0}$x_1\rightarrow x_2$}%
}}}}
\end{picture}%

&

\begin{picture}(0,0)%
\includegraphics{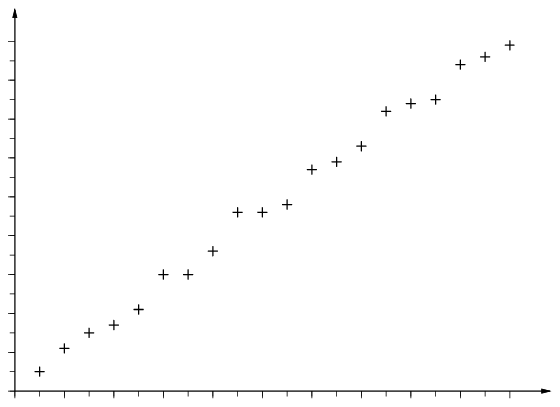}%
\end{picture}%
\setlength{\unitlength}{1302sp}%
\begingroup\makeatletter\ifx\SetFigFont\undefined%
\gdef\SetFigFont#1#2#3#4#5{%
  \reset@font\fontsize{#1}{#2pt}%
  \fontfamily{#3}\fontseries{#4}\fontshape{#5}%
  \selectfont}%
\fi\endgroup%
\begin{picture}(8458,5952)(721,-5459)
\put(1156,-5410){\makebox(0,0)[lb]{\smash{{\SetFigFont{5}{6.0}{\familydefault}{\mddefault}{\updefault}{\color[rgb]{0,0,0}$0$}%
}}}}
\put(1876,-5410){\makebox(0,0)[lb]{\smash{{\SetFigFont{5}{6.0}{\familydefault}{\mddefault}{\updefault}{\color[rgb]{0,0,0}$2$}%
}}}}
\put(2596,-5410){\makebox(0,0)[lb]{\smash{{\SetFigFont{5}{6.0}{\familydefault}{\mddefault}{\updefault}{\color[rgb]{0,0,0}$4$}%
}}}}
\put(3316,-5410){\makebox(0,0)[lb]{\smash{{\SetFigFont{5}{6.0}{\familydefault}{\mddefault}{\updefault}{\color[rgb]{0,0,0}$6$}%
}}}}
\put(4036,-5410){\makebox(0,0)[lb]{\smash{{\SetFigFont{5}{6.0}{\familydefault}{\mddefault}{\updefault}{\color[rgb]{0,0,0}$8$}%
}}}}
\put(4710,-5410){\makebox(0,0)[lb]{\smash{{\SetFigFont{5}{6.0}{\familydefault}{\mddefault}{\updefault}{\color[rgb]{0,0,0}$10$}%
}}}}
\put(5430,-5410){\makebox(0,0)[lb]{\smash{{\SetFigFont{5}{6.0}{\familydefault}{\mddefault}{\updefault}{\color[rgb]{0,0,0}$12$}%
}}}}
\put(6150,-5410){\makebox(0,0)[lb]{\smash{{\SetFigFont{5}{6.0}{\familydefault}{\mddefault}{\updefault}{\color[rgb]{0,0,0}$14$}%
}}}}
\put(6870,-5410){\makebox(0,0)[lb]{\smash{{\SetFigFont{5}{6.0}{\familydefault}{\mddefault}{\updefault}{\color[rgb]{0,0,0}$16$}%
}}}}
\put(7590,-5410){\makebox(0,0)[lb]{\smash{{\SetFigFont{5}{6.0}{\familydefault}{\mddefault}{\updefault}{\color[rgb]{0,0,0}$18$}%
}}}}
\put(8310,-5410){\makebox(0,0)[lb]{\smash{{\SetFigFont{5}{6.0}{\familydefault}{\mddefault}{\updefault}{\color[rgb]{0,0,0}$20$}%
}}}}
\put(721,-4628){\makebox(0,0)[lb]{\smash{{\SetFigFont{5}{6.0}{\familydefault}{\mddefault}{\updefault}{\color[rgb]{0,0,0}$10$}%
}}}}
\put(721,-4062){\makebox(0,0)[lb]{\smash{{\SetFigFont{5}{6.0}{\familydefault}{\mddefault}{\updefault}{\color[rgb]{0,0,0}$20$}%
}}}}
\put(721,-3497){\makebox(0,0)[lb]{\smash{{\SetFigFont{5}{6.0}{\familydefault}{\mddefault}{\updefault}{\color[rgb]{0,0,0}$30$}%
}}}}
\put(721,-2932){\makebox(0,0)[lb]{\smash{{\SetFigFont{5}{6.0}{\familydefault}{\mddefault}{\updefault}{\color[rgb]{0,0,0}$40$}%
}}}}
\put(721,-2366){\makebox(0,0)[lb]{\smash{{\SetFigFont{5}{6.0}{\familydefault}{\mddefault}{\updefault}{\color[rgb]{0,0,0}$50$}%
}}}}
\put(721,-1801){\makebox(0,0)[lb]{\smash{{\SetFigFont{5}{6.0}{\familydefault}{\mddefault}{\updefault}{\color[rgb]{0,0,0}$60$}%
}}}}
\put(721,-1236){\makebox(0,0)[lb]{\smash{{\SetFigFont{5}{6.0}{\familydefault}{\mddefault}{\updefault}{\color[rgb]{0,0,0}$70$}%
}}}}
\put(721,-670){\makebox(0,0)[lb]{\smash{{\SetFigFont{5}{6.0}{\familydefault}{\mddefault}{\updefault}{\color[rgb]{0.000,0.000,0.000}$80$}%
}}}}
\put(721,-136){\makebox(0,0)[lb]{\smash{{\SetFigFont{5}{6.0}{\familydefault}{\mddefault}{\updefault}{\color[rgb]{0,0,0}$90$}%
}}}}
\put(901,-5193){\makebox(0,0)[lb]{\smash{{\SetFigFont{5}{6.0}{\familydefault}{\mddefault}{\updefault}{\color[rgb]{0,0,0}$0$}%
}}}}
\put(1351, 89){\makebox(0,0)[lb]{\smash{{\SetFigFont{8}{9.6}{\familydefault}{\mddefault}{\updefault}{\color[rgb]{0,0,0}$y_1(k)$}%
}}}}
\put(8626,-4936){\makebox(0,0)[lb]{\smash{{\SetFigFont{8}{9.6}{\familydefault}{\mddefault}{\updefault}{\color[rgb]{0,0,0}$k$}%
}}}}
\end{picture}%
 
\\

\begin{picture}(0,0)%
\includegraphics{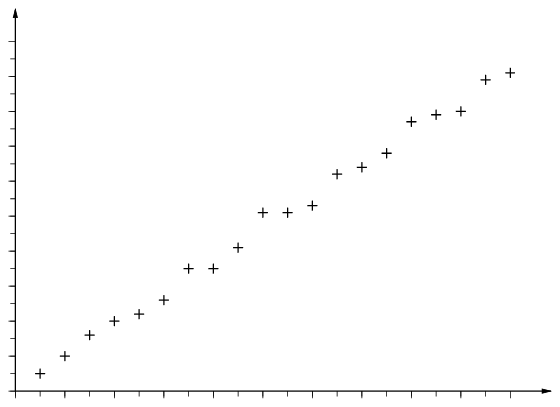}%
\end{picture}%
\setlength{\unitlength}{1302sp}%
\begingroup\makeatletter\ifx\SetFigFont\undefined%
\gdef\SetFigFont#1#2#3#4#5{%
  \reset@font\fontsize{#1}{#2pt}%
  \fontfamily{#3}\fontseries{#4}\fontshape{#5}%
  \selectfont}%
\fi\endgroup%
\begin{picture}(8638,5952)(541,-5459)
\put(1156,-5410){\makebox(0,0)[lb]{\smash{{\SetFigFont{5}{6.0}{\familydefault}{\mddefault}{\updefault}{\color[rgb]{0,0,0}$0$}%
}}}}
\put(1876,-5410){\makebox(0,0)[lb]{\smash{{\SetFigFont{5}{6.0}{\familydefault}{\mddefault}{\updefault}{\color[rgb]{0,0,0}$2$}%
}}}}
\put(2596,-5410){\makebox(0,0)[lb]{\smash{{\SetFigFont{5}{6.0}{\familydefault}{\mddefault}{\updefault}{\color[rgb]{0,0,0}$4$}%
}}}}
\put(3316,-5410){\makebox(0,0)[lb]{\smash{{\SetFigFont{5}{6.0}{\familydefault}{\mddefault}{\updefault}{\color[rgb]{0,0,0}$6$}%
}}}}
\put(4036,-5410){\makebox(0,0)[lb]{\smash{{\SetFigFont{5}{6.0}{\familydefault}{\mddefault}{\updefault}{\color[rgb]{0,0,0}$8$}%
}}}}
\put(4710,-5410){\makebox(0,0)[lb]{\smash{{\SetFigFont{5}{6.0}{\familydefault}{\mddefault}{\updefault}{\color[rgb]{0,0,0}$10$}%
}}}}
\put(5430,-5410){\makebox(0,0)[lb]{\smash{{\SetFigFont{5}{6.0}{\familydefault}{\mddefault}{\updefault}{\color[rgb]{0,0,0}$12$}%
}}}}
\put(6150,-5410){\makebox(0,0)[lb]{\smash{{\SetFigFont{5}{6.0}{\familydefault}{\mddefault}{\updefault}{\color[rgb]{0,0,0}$14$}%
}}}}
\put(6870,-5410){\makebox(0,0)[lb]{\smash{{\SetFigFont{5}{6.0}{\familydefault}{\mddefault}{\updefault}{\color[rgb]{0,0,0}$16$}%
}}}}
\put(7590,-5410){\makebox(0,0)[lb]{\smash{{\SetFigFont{5}{6.0}{\familydefault}{\mddefault}{\updefault}{\color[rgb]{0,0,0}$18$}%
}}}}
\put(8310,-5410){\makebox(0,0)[lb]{\smash{{\SetFigFont{5}{6.0}{\familydefault}{\mddefault}{\updefault}{\color[rgb]{0,0,0}$20$}%
}}}}
\put(901,-5193){\makebox(0,0)[lb]{\smash{{\SetFigFont{5}{6.0}{\familydefault}{\mddefault}{\updefault}{\color[rgb]{0,0,0}$0$}%
}}}}
\put(721,-4684){\makebox(0,0)[lb]{\smash{{\SetFigFont{5}{6.0}{\familydefault}{\mddefault}{\updefault}{\color[rgb]{0,0,0}$10$}%
}}}}
\put(721,-4175){\makebox(0,0)[lb]{\smash{{\SetFigFont{5}{6.0}{\familydefault}{\mddefault}{\updefault}{\color[rgb]{0,0,0}$20$}%
}}}}
\put(721,-3667){\makebox(0,0)[lb]{\smash{{\SetFigFont{5}{6.0}{\familydefault}{\mddefault}{\updefault}{\color[rgb]{0,0,0}$30$}%
}}}}
\put(721,-3158){\makebox(0,0)[lb]{\smash{{\SetFigFont{5}{6.0}{\familydefault}{\mddefault}{\updefault}{\color[rgb]{0,0,0}$40$}%
}}}}
\put(721,-2649){\makebox(0,0)[lb]{\smash{{\SetFigFont{5}{6.0}{\familydefault}{\mddefault}{\updefault}{\color[rgb]{0,0,0}$50$}%
}}}}
\put(721,-2140){\makebox(0,0)[lb]{\smash{{\SetFigFont{5}{6.0}{\familydefault}{\mddefault}{\updefault}{\color[rgb]{0,0,0}$60$}%
}}}}
\put(721,-1631){\makebox(0,0)[lb]{\smash{{\SetFigFont{5}{6.0}{\familydefault}{\mddefault}{\updefault}{\color[rgb]{0,0,0}$70$}%
}}}}
\put(721,-1123){\makebox(0,0)[lb]{\smash{{\SetFigFont{5}{6.0}{\familydefault}{\mddefault}{\updefault}{\color[rgb]{0,0,0}$80$}%
}}}}
\put(721,-614){\makebox(0,0)[lb]{\smash{{\SetFigFont{5}{6.0}{\familydefault}{\mddefault}{\updefault}{\color[rgb]{0,0,0}$90$}%
}}}}
\put(541,-105){\makebox(0,0)[lb]{\smash{{\SetFigFont{5}{6.0}{\familydefault}{\mddefault}{\updefault}{\color[rgb]{0,0,0}$100$}%
}}}}
\put(8626,-4936){\makebox(0,0)[lb]{\smash{{\SetFigFont{8}{9.6}{\familydefault}{\mddefault}{\updefault}{\color[rgb]{0,0,0}$k$}%
}}}}
\put(1351, 89){\makebox(0,0)[lb]{\smash{{\SetFigFont{8}{9.6}{\familydefault}{\mddefault}{\updefault}{\color[rgb]{0,0,0}$y_2(k)$}%
}}}}
\end{picture}%
 
&

\begin{picture}(0,0)%
\includegraphics{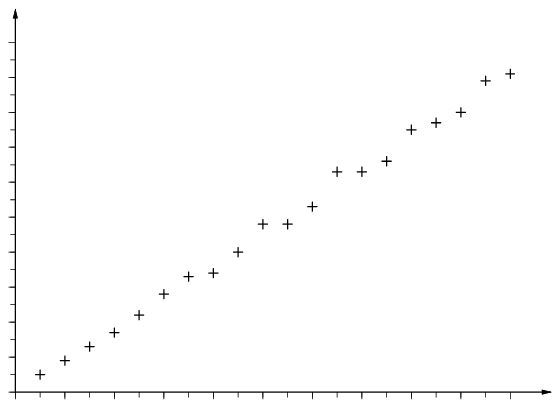}%
\end{picture}%
\setlength{\unitlength}{1302sp}%
\begingroup\makeatletter\ifx\SetFigFont\undefined%
\gdef\SetFigFont#1#2#3#4#5{%
  \reset@font\fontsize{#1}{#2pt}%
  \fontfamily{#3}\fontseries{#4}\fontshape{#5}%
  \selectfont}%
\fi\endgroup%
\begin{picture}(8563,5960)(541,-5459)
\put(1156,-5410){\makebox(0,0)[lb]{\smash{{\SetFigFont{5}{6.0}{\familydefault}{\mddefault}{\updefault}{\color[rgb]{0,0,0}$0$}%
}}}}
\put(1876,-5410){\makebox(0,0)[lb]{\smash{{\SetFigFont{5}{6.0}{\familydefault}{\mddefault}{\updefault}{\color[rgb]{0,0,0}$2$}%
}}}}
\put(2596,-5410){\makebox(0,0)[lb]{\smash{{\SetFigFont{5}{6.0}{\familydefault}{\mddefault}{\updefault}{\color[rgb]{0,0,0}$4$}%
}}}}
\put(3316,-5410){\makebox(0,0)[lb]{\smash{{\SetFigFont{5}{6.0}{\familydefault}{\mddefault}{\updefault}{\color[rgb]{0,0,0}$6$}%
}}}}
\put(4036,-5410){\makebox(0,0)[lb]{\smash{{\SetFigFont{5}{6.0}{\familydefault}{\mddefault}{\updefault}{\color[rgb]{0,0,0}$8$}%
}}}}
\put(4710,-5410){\makebox(0,0)[lb]{\smash{{\SetFigFont{5}{6.0}{\familydefault}{\mddefault}{\updefault}{\color[rgb]{0,0,0}$10$}%
}}}}
\put(5430,-5410){\makebox(0,0)[lb]{\smash{{\SetFigFont{5}{6.0}{\familydefault}{\mddefault}{\updefault}{\color[rgb]{0,0,0}$12$}%
}}}}
\put(6150,-5410){\makebox(0,0)[lb]{\smash{{\SetFigFont{5}{6.0}{\familydefault}{\mddefault}{\updefault}{\color[rgb]{0,0,0}$14$}%
}}}}
\put(6870,-5410){\makebox(0,0)[lb]{\smash{{\SetFigFont{5}{6.0}{\familydefault}{\mddefault}{\updefault}{\color[rgb]{0,0,0}$16$}%
}}}}
\put(7590,-5410){\makebox(0,0)[lb]{\smash{{\SetFigFont{5}{6.0}{\familydefault}{\mddefault}{\updefault}{\color[rgb]{0,0,0}$18$}%
}}}}
\put(8310,-5410){\makebox(0,0)[lb]{\smash{{\SetFigFont{5}{6.0}{\familydefault}{\mddefault}{\updefault}{\color[rgb]{0,0,0}$20$}%
}}}}
\put(901,-5193){\makebox(0,0)[lb]{\smash{{\SetFigFont{5}{6.0}{\familydefault}{\mddefault}{\updefault}{\color[rgb]{0,0,0}$0$}%
}}}}
\put(721,-4684){\makebox(0,0)[lb]{\smash{{\SetFigFont{5}{6.0}{\familydefault}{\mddefault}{\updefault}{\color[rgb]{0,0,0}$10$}%
}}}}
\put(721,-4175){\makebox(0,0)[lb]{\smash{{\SetFigFont{5}{6.0}{\familydefault}{\mddefault}{\updefault}{\color[rgb]{0,0,0}$20$}%
}}}}
\put(721,-3667){\makebox(0,0)[lb]{\smash{{\SetFigFont{5}{6.0}{\familydefault}{\mddefault}{\updefault}{\color[rgb]{0,0,0}$30$}%
}}}}
\put(721,-3158){\makebox(0,0)[lb]{\smash{{\SetFigFont{5}{6.0}{\familydefault}{\mddefault}{\updefault}{\color[rgb]{0,0,0}$40$}%
}}}}
\put(721,-2649){\makebox(0,0)[lb]{\smash{{\SetFigFont{5}{6.0}{\familydefault}{\mddefault}{\updefault}{\color[rgb]{0,0,0}$50$}%
}}}}
\put(721,-2140){\makebox(0,0)[lb]{\smash{{\SetFigFont{5}{6.0}{\familydefault}{\mddefault}{\updefault}{\color[rgb]{0,0,0}$60$}%
}}}}
\put(721,-1631){\makebox(0,0)[lb]{\smash{{\SetFigFont{5}{6.0}{\familydefault}{\mddefault}{\updefault}{\color[rgb]{0,0,0}$70$}%
}}}}
\put(721,-1123){\makebox(0,0)[lb]{\smash{{\SetFigFont{5}{6.0}{\familydefault}{\mddefault}{\updefault}{\color[rgb]{0,0,0}$80$}%
}}}}
\put(721,-614){\makebox(0,0)[lb]{\smash{{\SetFigFont{5}{6.0}{\familydefault}{\mddefault}{\updefault}{\color[rgb]{0,0,0}$90$}%
}}}}
\put(541,-105){\makebox(0,0)[lb]{\smash{{\SetFigFont{5}{6.0}{\familydefault}{\mddefault}{\updefault}{\color[rgb]{0,0,0}$100$}%
}}}}
\put(8551,-4936){\makebox(0,0)[lb]{\smash{{\SetFigFont{8}{9.6}{\familydefault}{\mddefault}{\updefault}{\color[rgb]{0,0,0}$k$}%
}}}}
\put(1351, 89){\makebox(0,0)[lb]{\smash{{\SetFigFont{8}{9.6}{\familydefault}{\mddefault}{\updefault}{\color[rgb]{0,0,0}$y_3(k)$}%
}}}}
\end{picture}%

\end{tabular}
\caption{Random holding times and the corresponding output trajectory}
\label{figure2}
\end{center}
\end{figure}

In Figure~\ref{figure2} we represented the output trajectory 
(corresponding to the initial condition $x_i(0)=\unit$, for $i=1,\dots ,9$) 
of the usual description, that is through a max-plus linear dynamical system 
of the form
\[
\left\{ 
\begin{array}{l}
x(k+1)=\bar{A} x(k)   \\
y(k)=C x(k) 
\end{array}
\right. \; ,
\]
of the flow-shop of Figure~\ref{figure1}, 
when the uncertain holding times take the values given on the upper left-hand 
side of Figure~\ref{figure2}. These times have been generated at random, 
in their respective intervals, using Scilab. 
With this output trajectory and the initial condition, 
we computed the sequence $\left\{z(k)\right\}_{k\in \N}$ 
given by the dynamic observer~\eqref{EqObserver}. In particular, 
in Figure~\ref{figure3} we represented the sequence 
$\left\{z_3(k)\right\}_{k\in \N}$  which is equal to the 
sequence  $\left\{x_7(k)\right\}_{k\in \N}$ of firing times of the 
seventh transition $x_7$ because $z_3(k)=x_7(k)$ by the form of $F$.  

\begin{figure}
\begin{center}

\begin{picture}(0,0)%
\includegraphics{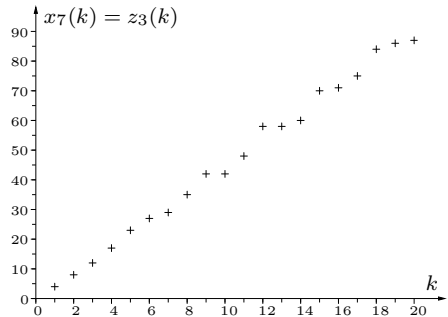}%
\end{picture}%
\setlength{\unitlength}{1302sp}%
\begingroup\makeatletter\ifx\SetFigFont\undefined%
\gdef\SetFigFont#1#2#3#4#5{%
  \reset@font\fontsize{#1}{#2pt}%
  \fontfamily{#3}\fontseries{#4}\fontshape{#5}%
  \selectfont}%
\fi\endgroup%
\begin{picture}(8458,5952)(721,-5459)
\put(1156,-5410){\makebox(0,0)[lb]{\smash{{\SetFigFont{5}{6.0}{\familydefault}{\mddefault}{\updefault}{\color[rgb]{0,0,0}$0$}%
}}}}
\put(1876,-5410){\makebox(0,0)[lb]{\smash{{\SetFigFont{5}{6.0}{\familydefault}{\mddefault}{\updefault}{\color[rgb]{0,0,0}$2$}%
}}}}
\put(2596,-5410){\makebox(0,0)[lb]{\smash{{\SetFigFont{5}{6.0}{\familydefault}{\mddefault}{\updefault}{\color[rgb]{0,0,0}$4$}%
}}}}
\put(3316,-5410){\makebox(0,0)[lb]{\smash{{\SetFigFont{5}{6.0}{\familydefault}{\mddefault}{\updefault}{\color[rgb]{0,0,0}$6$}%
}}}}
\put(4036,-5410){\makebox(0,0)[lb]{\smash{{\SetFigFont{5}{6.0}{\familydefault}{\mddefault}{\updefault}{\color[rgb]{0,0,0}$8$}%
}}}}
\put(4710,-5410){\makebox(0,0)[lb]{\smash{{\SetFigFont{5}{6.0}{\familydefault}{\mddefault}{\updefault}{\color[rgb]{0,0,0}$10$}%
}}}}
\put(5430,-5410){\makebox(0,0)[lb]{\smash{{\SetFigFont{5}{6.0}{\familydefault}{\mddefault}{\updefault}{\color[rgb]{0,0,0}$12$}%
}}}}
\put(6150,-5410){\makebox(0,0)[lb]{\smash{{\SetFigFont{5}{6.0}{\familydefault}{\mddefault}{\updefault}{\color[rgb]{0,0,0}$14$}%
}}}}
\put(6870,-5410){\makebox(0,0)[lb]{\smash{{\SetFigFont{5}{6.0}{\familydefault}{\mddefault}{\updefault}{\color[rgb]{0,0,0}$16$}%
}}}}
\put(7590,-5410){\makebox(0,0)[lb]{\smash{{\SetFigFont{5}{6.0}{\familydefault}{\mddefault}{\updefault}{\color[rgb]{0,0,0}$18$}%
}}}}
\put(8310,-5410){\makebox(0,0)[lb]{\smash{{\SetFigFont{5}{6.0}{\familydefault}{\mddefault}{\updefault}{\color[rgb]{0,0,0}$20$}%
}}}}
\put(901,-5193){\makebox(0,0)[lb]{\smash{{\SetFigFont{5}{6.0}{\familydefault}{\mddefault}{\updefault}{\color[rgb]{0,0,0}$0$}%
}}}}
\put(721,-4628){\makebox(0,0)[lb]{\smash{{\SetFigFont{5}{6.0}{\familydefault}{\mddefault}{\updefault}{\color[rgb]{0,0,0}$10$}%
}}}}
\put(721,-4062){\makebox(0,0)[lb]{\smash{{\SetFigFont{5}{6.0}{\familydefault}{\mddefault}{\updefault}{\color[rgb]{0,0,0}$20$}%
}}}}
\put(721,-3497){\makebox(0,0)[lb]{\smash{{\SetFigFont{5}{6.0}{\familydefault}{\mddefault}{\updefault}{\color[rgb]{0,0,0}$30$}%
}}}}
\put(721,-2932){\makebox(0,0)[lb]{\smash{{\SetFigFont{5}{6.0}{\familydefault}{\mddefault}{\updefault}{\color[rgb]{0,0,0}$40$}%
}}}}
\put(721,-2366){\makebox(0,0)[lb]{\smash{{\SetFigFont{5}{6.0}{\familydefault}{\mddefault}{\updefault}{\color[rgb]{0,0,0}$50$}%
}}}}
\put(721,-1801){\makebox(0,0)[lb]{\smash{{\SetFigFont{5}{6.0}{\familydefault}{\mddefault}{\updefault}{\color[rgb]{0,0,0}$60$}%
}}}}
\put(721,-1236){\makebox(0,0)[lb]{\smash{{\SetFigFont{5}{6.0}{\familydefault}{\mddefault}{\updefault}{\color[rgb]{0,0,0}$70$}%
}}}}
\put(721,-670){\makebox(0,0)[lb]{\smash{{\SetFigFont{5}{6.0}{\familydefault}{\mddefault}{\updefault}{\color[rgb]{0,0,0}$80$}%
}}}}
\put(721,-105){\makebox(0,0)[lb]{\smash{{\SetFigFont{5}{6.0}{\familydefault}{\mddefault}{\updefault}{\color[rgb]{0,0,0}$90$}%
}}}}
\put(1351,164){\makebox(0,0)[lb]{\smash{{\SetFigFont{8}{9.6}{\familydefault}{\mddefault}{\updefault}{\color[rgb]{0,0,0}$x_7(k)=z_3(k)$}%
}}}}
\put(8626,-4936){\makebox(0,0)[lb]{\smash{{\SetFigFont{8}{9.6}{\familydefault}{\mddefault}{\updefault}{\color[rgb]{0,0,0}$k$}%
}}}}
\end{picture}%

\caption{The (exact) state trajectory reconstructed by the observer}
\label{figure3}
\end{center}
\end{figure}

Indeed, it is possible to directly check the dynamic 
observer~\eqref{EqObserver}. Assume that $z(k)=F x(k)$. Then, 
for instance, by~\eqref{EqObserver} we have 
\[
z_3(k+1) = 4 z_2(k)\oplus 3 z_4(k)= 4 x_4(k)\oplus 3 x_9(k) \; , 
\]
because $z_2(k)= x_4(k)$ and $z_4(k)=x_9(k)$, 
and by~\eqref{EqEvSystem} we know that 
\[
x_7(k+1) = 4 x_4(k)\oplus 3 x_9(k) \; . 
\]
Therefore, $z_3(k+1)=x_7(k+1)$.

Let us finally mention that Timed Event Graphs in which the number 
of initial tokens and holding times are only known to belong to certain 
intervals have been considered in~\cite{LHCJ04}. This work addresses the 
existence and computation of a robust control set in order to guarantee that 
the output of the controlled system is contained in a set of reference 
outputs. In contrast to the present paper, it is based on interval 
analysis in dioids, residuation theory and transfer series methods, 
and does not address any observation problem.

\end{document}